\documentclass[twoside,a4paper,reqno,11pt]{amsart}
\usepackage{amsfonts, amsbsy, amsmath, amsthm, amssymb, latexsym, verbatim, enumerate}
\usepackage{mathrsfs,hyperref}
\usepackage[top=30mm,right=30mm,bottom=30mm,left=30mm]{geometry}

\usepackage{bm}
\usepackage[pdftex]{color,graphicx}

\makeatletter
\newcommand{\imod}[1]{\allowbreak\mkern4mu({\operator@font mod}\,\,#1)}
\makeatother

\numberwithin{equation}{section}

\def\a{\alpha}
 \def\b{\beta}
 \def\e{\epsilon}
 \def\d{\delta}
  \def\D{\Delta}

 \def\c{\chi}
 
 \def\g{\gamma}

 \def\N{\mathbb N}
 
 \def\Z{\mathbb Z}
 \def\C{\mathbb C}
 \def\F{\mathbb F}

 \def\Cl{{\rm Cl}}

 \def\l{\lambda}
\def\D{\Delta}
 
 \def\s{\sigma}
 \def\O{\Omega}

 \def\bmax{\mathsf{b}}
 
 \def\St{\mathsf{St}}
 \def\SSS{\mathsf{S}}
 \def\AAA{\mathsf{A}}
 \def\Irr{\mathrm{Irr}}
 \def\Ind{\mathrm{Ind}}
 \def\Ker{\mathrm{Ker}}
 \def\supp{\mathsf{supp}}
 \def\Stab{\mathrm{Stab}}
 \def\Z{\mathbb{Z}}
 \def\GSS{G_{\mathrm {ss}}}
 \def\CB{\mathbf{C}}
 \def\ZB{\mathbf{Z}}
 \def\NB{\mathbf{N}}
 
 \def\SL{\mathrm {SL}}
 \def\GL{\mathrm {GL}}
 \def\PSL{\mathrm {PSL}}
 \def\GU{\mathrm {GU}}

 \def\Sp{\mathrm {Sp}}

 \def\SO{\mathrm {SO}}
  \def\GO{\mathrm {O}}
 \def\MC{\mathcal {M}}
 \def\XC{\mathcal {X}}
 \def\QF{\mathsf{Q}}

 \newcommand{\tw}[1]{{}^#1\!}
\renewcommand{\mod}{\bmod \,}

\newtheorem{theorem}{Theorem}
\newtheorem*{conj*}{Conjecture}

\newtheorem{conj}[theorem]{Conjecture}

\newtheorem{thm}{Theorem}[section]
\newtheorem{prop}[thm]{Proposition}
\newtheorem{lem}[thm]{Lemma}

\theoremstyle{definition}

\begin{document}

 \author{Martin W. Liebeck}
\address{M.W. Liebeck, Department of Mathematics,
    Imperial College, London SW7 2BZ, UK}
\email{m.liebeck@imperial.ac.uk}

\author{Aner Shalev}
\address{A. Shalev, Institute of Mathematics, Hebrew University, Jerusalem 91904, Israel}
\email{shalev@math.huji.ac.il}

\author[P. H. Tiep]{Pham Huu Tiep}
\address{P.H. Tiep, Department of Mathematics, Rutgers University, Piscataway, NJ 08854, USA}
\email{tiep@math.rutgers.edu}

\title{McKay graphs for alternating and classical groups}

\maketitle

\begin{abstract}
Let $G$ be a finite group, and $\a$ a nontrivial  character of $G$. The McKay graph $\MC(G,\a)$ has the irreducible characters of $G$ as vertices, with an edge from $\c_1$ to $\c_2$ if $\c_2$ is a constituent of $\a\c_1$. We study the diameters of McKay graphs for finite simple groups $G$. For alternating groups, we prove a conjecture made in \cite{LST}: there is an absolute constant $C$ such that $\hbox{diam}\,{\mathcal M}(G,\a) \le C\frac{\log |\AAA_n|}{\log \a(1)}$ for all nontrivial irreducible characters $\a$ of $\AAA_n$. Also for classsical groups of symplectic or orthogonal type of rank $r$, we establish a linear upper bound $Cr$ on the diameters of all nontrivial McKay graphs. 
 \end{abstract}


\date{\today}


\footnote{The second author acknowledges the support of ISF grant 686/17, and the Vinik chair of mathematics which he holds. 
The third author gratefully acknowledges the support of the
NSF (grant DMS-1840702), and the Joshua Barlaz Chair in Mathematics. The second and the third authors 
were partially supported by BSF grant 2016072.
The authors also acknowledge the support of the National Science Foundation under Grant No. DMS-1440140 while they were in residence at the Mathematical Sciences Research Institute in Berkeley, California, during the Spring 2018 semester. Part of this work was done when
the authors were in residence at the Isaac Newton Institute for Mathematical Sciences, Cambridge, UK, in Spring 2020, 
and partially supported by a grant from the Simons Foundation.}


\section{Introduction}

For a finite group $G$, and a (complex) character $\a$ of $G$, the {\it McKay graph} $\MC(G,\a)$ is defined to be the directed graph with vertex set ${\rm Irr}(G)$, there being an edge from $\c_1$ to $\c_2$ if and only if $\c_2$ is a constituent of $\a\c_1$.
A classical result of Burnside and Brauer \cite{Br} shows that $\MC(G,\a)$ is connected if and only if $\a$ is faithful. 

The study of McKay graphs for finite simple groups $G$ was initiated in \cite{LST}, with a particular focus on the diameters of these graphs. Theorem 2 of \cite{LST} establishes a quadratic upper bound $\hbox{diam}\,{\mathcal M}(G,\a) \le Cr^2$ for any simple group $G$ of Lie type or rank $r$ and any nontrivial $\a \in {\rm Irr}(G)$. Notice that  the smallest (resp. largest) nontrivial irreducible character degrees of $G$ are at most $q^{cr}$ (resp. at least $q^{c'r^2}$), where $c,c'$ are constants, and hence the maximal diameter of a McKay graph  ${\mathcal M}(G,\a)$ is at least a linear function of $r$. Theorem 3 of \cite{LST} implies a linear upper bound on these diameters for the classical groups $G=\PSL_n^\e(q)$, provided $q$ is large compared to $n$. Our first main result establishes a linear upper bound for the remaining classical groups.

\begin{theorem}\label{main1}
 Let $G$ be a quasisimple classical group $Sp_n(q)$ or $\O_n^\e(q)$, and let $\a$ be a nontrivial irreducible character of $G$. Then $\hbox{diam}\,{\mathcal M}(G,\a) \le Cn$, where $C=16$ or $32$, respectively.
\end{theorem}

An obvious lower bound for $\hbox{diam}\,{\mathcal M}(G,\a)$ (when $\a(1)>1$) is given by 
$\frac{\log \bmax(G)}{\log \a(1)}$, where $\bmax(G)$ is the largest degree of an irreducible character of $G$. In \cite[Conjecture 1]{LST} we conjectured that for simple groups $G$, this bound is tight up to a multiplicative constant. This conjecture was proved in \cite[Theorem 3]{LST} for the simple groups $\PSL_n^\e(q)$, provided $q$ is large compared to $n$. Recently it has also been established for the symmetric groups in \cite{S}. Deducing it for the alternating groups is not entirely trivial, and this is the content of our next result.

\begin{theorem}\label{main2}
There is an effective absolute constant $C$ such that, for all $n \geq 5$ and for all nontrivial irreducible characters $\a$ of  $G:=\AAA_n$,
$$\hbox{diam}\,{\mathcal M}(G,\a) \le C\frac{\log |G|}{\log \a(1)}.$$
\end{theorem}

In our final result, we consider covering ${\rm Irr}(G)$ by products of arbitrary irreducible characters, instead of powers of a fixed character. This idea was suggested by Gill \cite{G}, inspired by an analogous result of Rodgers and Saxl \cite{RS} for conjugacy classes in $G=\SL_n(q)$: this states that if a collection of conjugacy classes of $G$ satisfies the condition that the product of the class sizes is at least $|G|^{12}$, then the product of the classes is equal to $G$.

As a piece of notation, for characters $\c_1,\ldots,\c_l$ of $G$, we write $\c_1\c_2\cdots \c_l \supseteq \Irr(G)$ to mean that every irreducible character of $G$ appears as a constituent of $\c_1\c_2\cdots \c_l$. Also, let $g: \N\to \N$ be the function appearing in \cite[Theorem 3]{LST}.

\begin{theorem}\label{rodsax} 
\begin{itemize}
\item[{\rm (i)}]  Let $G$ be a simple group of Lie type of rank $r$, let $l \ge 489r^2$, and let $\c_1,\ldots,\c_l \in \Irr(G) \setminus 1_G$. Then $\c_1\c_2\cdots \c_l \supseteq \Irr(G)$.

\vspace{2mm}
\item[{\rm (ii)}] Let $G = \PSL_n^\e(q)$ with $q>g(n)$, let $l \in \N$, and let $\c_1,\ldots \c_l \in \Irr(G)$ satisfy $\prod_1^l \c_i(1) > |G|^{10}$. Then $\c_1\c_2\cdots \c_l \supseteq \Irr(G)$.
\end{itemize}
\end{theorem}

Gill \cite{G} has conjectured that part (ii) of the theorem holds for all simple groups (with the constant 10 possibly replaced by a different constant). As a stepping stone in the spirit of the linear bound given by Theorem \ref{main1}, let us pose the following more modest conjecture.

\begin{conj}\label{rsax} There is an absolute constant $C>0$ such that the following holds. Let $G=\Cl_n(q)$, a classical simple  group of dimension $n$, or $\AAA_n$, an alternating group of degree $n\ge 5$. Let $l \ge Cn$,  and let $\c_1,\ldots,\c_l \in \Irr(G) \setminus 1_G$. Then $\c_1\c_2\cdots \c_l \supseteq \Irr(G)$.
\end{conj}

See Proposition \ref{rs2-an} for some partial result on Conjecture \ref{rsax} in the cae of $\AAA_n$.

The layout of the paper is as follows. Section \ref{prel1} contains a substantial amount of character theory for symplectic and orthogonal groups that is required for the proof of Theorem \ref{main1}, which is completed in Section \ref{pfth1}. The remaining sections \ref{pfth2} and \ref{pfth3} contain the proofs of Theorems \ref{main2} and \ref{rodsax}, respectively.

\section{Some character theory for symplectic and orthogonal groups}\label{prel1}

Let $V = \F_q^d$ be endowed with a non-degenerate, alternating or quadratic of type $\e = \pm$, form and 
let $G$ denote the derived subgroup of the full isometry group of the form. Assume that $G$ is quasisimple, so that $G = \Sp(V) = \Sp_d(q)$ or $\O(V) = \O^\e_d(q)$. 

This section contains a detailed study of some specific irreducible characters $\c$ of $G$ -- namely, the constituents of the permutation character $\Ind^G_{[P,P]}(1_{[P,P]})$, where $P$ is the maximal parabolic subgroup of $G$ stabilizing a singular 1-space. Two of the main results of the section are Propositions \ref{rat-so21} and \ref{rat-sp-so22}, which give upper bounds for the character ratios $|\c(g)/\c(1)|$ for $g\in G$. These will be used in Section \ref{pfth1} to prove Theorem \ref{main1}.

\subsection{Reduction lemmas}\label{red}
It is well known that the permutation action of $G$ on the set of singular $1$-spaces of 
$V$ is primitive of rank $3$, and thus its character is $\rho = 1_G + \a + \b$, with $\a, \b \in \Irr(G)$. Let (the parabolic subgroup) $P=QL$ denote a point stabilizer in this action, with $Q$ the unipotent radical and $L$ a Levi subgroup.
Aside from $\a,\b$, we also need to consider the remaining non-principal irreducible 
constituents $\g_i$ of $\Ind^G_{[P,P]}(1_{[P,P]})$. Let $\St$ denote the 
Steinberg character of $G$.

\begin{lem}\label{mc-r1}
The following statements hold.
\begin{enumerate}[\rm(i)]
\item Suppose that every semisimple element $s \in G$ is real. Then for any $\chi \in \Irr(G)$ and $k \in \N$, $\chi^{2k}$ contains
$\St$ if and only if $(\c\overline\c)^k$ contains $\St$.
\item All semisimple elements in $G$ are real, if $G = \Sp_{2n}(q)$, $\O_{2n+1}(q)$, or $\O^\e_{4n}(q)$.
\end{enumerate}
\end{lem}

\begin{proof}
(i) Recall that $\St(g) = 0$ if $g \in G$ is not semisimple. Furthermore, $\c(g) = \overline\c(g)$ if $g \in G$ is semisimple, by hypothesis. Hence 
$$\begin{aligned}
   ~[\chi^{2k},\St]_G & = \frac{1}{|G|}\sum_{g \in G}\chi(g)^{2k}\overline\St(g)\\ 
   & = \frac{1}{|G|}\sum_{g \in G,~g\mbox{ {\tiny semisimple}}}\chi(g)^{2k}\overline\St(g)\\
   & =  \frac{1}{|G|}\sum_{g \in G,~g\mbox{ {\tiny semisimple}}}\chi(g)^{k}\overline\c(g)^k\overline\St(g)\\
   & = \frac{1}{|G|}\sum_{g \in G}\c(g)^k\overline\c(g)^{k}\overline\St(g) = [(\c\overline\c)^k,\St]_G,
   \end{aligned}$$
and the claim follows.   

\smallskip
(ii) This is well known, see e.g. \cite[Proposition 3.1]{TZ2}.
\end{proof}

\begin{lem}\label{mc-r2}
Let $G = \Sp(V) = \Sp_{2n}(q)$ with $n \geq 3$. Suppose $C \in \N$ is such that both $\a^C$ and $\b^C$ contain $\St$.
Then for any $1_G \neq \chi \in \Irr(G)$, $\c^{2C}$ contains $\St$.
\end{lem}

\begin{proof}
In the aforementioned rank $3$ permutation action of $G$ with character $\rho = 1_G+\a+\b$, a point stabilizer 
$P$ is the normalizer $\NB_G(Z)$ of some long-root subgroup $Z$. Since $n \geq 3$, $Z$ has a nonzero fixed point on any
$\C G$-module affording $\c$ by \cite[Theorem 1.6]{T}. It follows that $\c|_P$ is reducible, and so
\begin{equation}\label{eq:mc1}
  2 \leq [\c|_P,\c|_P]_P = [\c\overline\c,\Ind^G_P(1_P)]_G = [\c\overline\c,\rho]_G.
\end{equation}  
As $[\c\overline\c,1_G]_G = 1$, $\c\overline\c$ contains either $\a$ or $\b$, whence $(\c\overline\c)^C$ contains $\St$. 
Applying Lemma \ref{mc-r1}, we conclude that $\c^{2C}$ contains $\St$. 
\end{proof}

\begin{lem}\label{mc-r3}
Let $G = \O(V) = \O^\e_{n}(q)$ with $n \geq 5$. Suppose $C \in \N$ is such that both $\a^C$ and $\b^C$ contain $\St$.
Consider any $1_G \neq \chi \in \Irr(G)$, and suppose in addition that either $n \not\equiv 2 (\bmod\ 4)$, or $\c = \overline\c$.
Then $\c^{4C}$ contains $\St$.
\end{lem}

\begin{proof}
Again we consider a point stabilizer $P=QL$ in the aforementioned rank $3$ permutation action of $G$ with character 
$\rho = 1_G+\a+\b$. 
Note that $Q$ is elementary abelian, $[L,L] \cong \O^\e_{n-2}(q)$, and we can identify $\Irr(Q)$ with the natural module 
$\F_q^{n-2}$ for $[L,L]$.  In particular, any $[L,L]$-orbit on $\Irr(Q) \smallsetminus \{1_Q\}$ has length at least $2$. It is also 
clear that some irreducible constituent of $\c|_Q$ is non-principal, since $\Ker(\c) \leq \ZB(G)$ and $Q \not\leq \ZB(G)$. It follows
that $\c|_Q$ is reducible, and so 
$$2 \leq [\c|_Q,\c|_Q]_Q = [(\c\overline\c)|_Q,1_Q]_Q.$$
Since $[\c\overline\c,1_G]_G = 1$, at least one non-principal irreducible constituent $\theta$ of $\c\overline\c$ contains $1_Q$ on restriction to $Q$. But $P$ normalizes $Q$, so the latter implies that $\theta|_P$ is reducible. Thus 
\eqref{eq:mc1} holds for $\theta$ instead of $\c$. Arguing as in the proof of Lemma \ref{mc-r1}, we obtain that 
$\theta\overline\theta$ contains either $\a$ or $\b$, whence $(\c\overline\c)^2$ contains either $\a$ or $\b$.
It follows that $(\c\overline\c)^{2C}$ contains $\St$, and we are done if $\c = \overline\c$.
Applying Lemma \ref{mc-r1}, we also have that $\c^{4C}$ contains $\St$ in the case $n \not\equiv 2 (\bmod\ 4)$. 
\end{proof}

\begin{lem}\label{mc-r4}
Let $G = \O(V) = \O^\e_{n}(q)$ with $n \geq 10$ and $n \equiv 2 (\bmod\ 4)$. Suppose $C \in \N$ is such that each of $\a^C$, 
$\b^C$, and $\g_i^C$ contains $\St$. Then for any $\chi \in \Irr(G)$ with $\c \neq \overline\c$, $\c^{4C}$ contains $\St$.
\end{lem}

\begin{proof}
(i) As noted in the proof of Lemma \ref{mc-r3}, $Q$ is elementary abelian, $[L,L] \cong \O^\e_{n-2}(q)$, and we can identify 
$\Irr(Q)$ with the natural module $\F_q^{n-2}$ for $[L,L]$. Since $n-2 \geq 8$, it is straightforward to check that any 
$[L,L]$-orbit on nonzero vectors of $\F_q^{n-2}$ contains a vector $v$ and also $-v$. 
Thus, any $[L,L]$-orbit on $\Irr(Q) \smallsetminus \{1_Q\}$ contains a characters $\l$ and also its complex conjugate 
$\overline\l$. As noted in the proof of Lemma \ref{mc-r3}, $Q \not\leq \Ker(\c)$. Thus we may assume that $\c|_Q$ contains 
$\l$ and also $\overline\l$. It follows that 
$1 \leq [\c^2|_Q,1_Q]_Q$.  Since $[\c^2,1_G]_G = [\c,\overline\c]_G = 0$, at least one non-principal irreducible constituent 
$\theta$ of $\c^2$ contains $1_Q$ on restriction to $Q$. 

In particular, $\theta|_P$ is reducible, since $P$ normalizes
$Q$, and \eqref{eq:mc1} holds for $\theta$ instead of $\c$, and so the 
arguments in the proof of Lemma \ref{mc-r2} shows that $\theta\overline\theta$ contains $\a$ or 
$\b$. If, moreover, $\theta = \overline\theta$, then we conclude that $\theta^2$ contains $\a$ or $\b$.

\smallskip
(ii) Now consider the case $\theta \neq \overline\theta$, and 
let $\theta$ be afforded by a $\C G$-module $U$. As shown in (i), the $Q$-fixed point subspace $U^Q$ on $U$ is 
nonzero, and $L$ acts on $U^Q$. Recall that $4|(n-2)$ and $n-2 \geq 8$. Now, if $(\e,q) \neq (+, \equiv 3(\bmod\ 4))$, then all irreducible characters of $[L,L] \cong \O^e_{n-2}(q)$ are real-valued, and so the $[L,L]$-module $U^Q$ contains an irreducible submodule $W \cong W^*$. 

Consider the case $(\e,q) = (+,\equiv 3(\bmod\ 4))$ and let $P = \Stab_G(\langle u \rangle_{\F_q})$ for 
a singular vector $0 \neq u \in V$. We can consider $P$ inside $\tilde P:=\Stab_{\SO(V)}(\langle u \rangle_{\F_q})=Q\tilde L$, 
and find another singular vector $u' \in V$ such that $V = V_1 \oplus V_2$, with $V_1 = \langle u,u' \rangle_{\F_q}$,
$V_2 = V_1^{\perp}$, and $[L,L] = \O(V_2)$. Since $q \equiv 3 (\bmod\ 4)$, $t:=-1_{V_1} \in \SO(V_1) \smallsetminus \O(V_1)$.
Choosing some $t' \in \SO(V_2) \smallsetminus \O(V_2)$, we see that $tt' \in \tilde L \cap \O(V) = L$, and 
$L_1 := \langle [L,L],tt' \rangle \cong \SO^+_{n-2}(q)$. By \cite{Gow}, all irreducible characters of $L_1$ are real-valued, and so the $L_1$-module $U^Q$ contains an irreducible submodule $W \cong W^*$. 

We have shown that the $[L,L]$-module $U^Q$ contains a nonzero submodule $W \cong W^*$. We can also inflate
$W$ to a nonzero self-dual module over $[P,P] = Q[L,L]$. It follows that $(U \otimes_{\C} U)|_{[P,P]}$ contains 
$W \otimes_{\C} W^*$, which certainly contains the trivial submodule. Thus, $\theta^2|_{[P,P]}$ contains the principal
character $1_{[P,P]}$, and so 
\begin{equation}\label{eq:mc2}
  1 \leq [\theta^2,\Ind^G_{[P,P]}(1_{[P,P]})]_G.
\end{equation}   
Recall we are assuming that $0 = [\theta,\overline\theta]_G = [\theta^2,1_G]_G$. Hence \eqref{eq:mc2} implies that 
$\theta^2$ contains at least one of $\a$, $\b$, or $\g_i$.

\smallskip
(iii) We have shown that, in all cases, $\theta^2$ contains at least one of $\a$, $\b$, or $\g_i$. As $\chi^2$ contains $\theta$,
we see that $\c^4$ contains at least one of $\a$, $\b$, or $\g_i$, and so $\c^{4C}$ contains $\St^2$.
\end{proof}

\subsection{Classical groups in characteristic $2$}
In this subsection we study certain characters of $\tilde G = \Sp(V) = \Sp_{2n}(q)$ and $G = \O(V)=\O^\e_{2n}(q)$, 
where $n \geq 5$ and $2|q$. These results will be used subsequently and are also of independent interest.

First we endow $V$ with a non-degenerate alternating form $(\cdot,\cdot)$, and work with its isometry group 
$\tilde G = \Sp(V)$. We will consider the following irreducible characters of $\tilde G$:

$\bullet$ the $q/2+1$ {\it linear-Weil} characters: 
$\rho^1_n$ of degree $(q^n+1)(q^n-q)/2(q-1)$, $\rho^2_n$ of 
degree $(q^n-1)(q^n+q)/2(q-1)$, and $\tau^i_n$ of degree $(q^{2n}-1)/(q-1)$, $1 \leq i \leq (q-2)/2$, and 

$\bullet$ the $q/2+2$ {\it unitary-Weil} characters:
$\a_n$ of degree $(q^n-1)(q^n-q)/2(q+1)$, $\b_n$ of 
degree $(q^n+1)(q^n+q)/2(q+1)$, and $\zeta^i_n$ of degree $(q^{2n}-1)/(q+1)$, $1 \leq i \leq q/2$;\\ 
see \cite[Table 1]{GT}. Then 
\begin{equation}\label{eq:dec11}
  \rho:=1_{\tilde G}+\rho^1_n+\rho^2_n
\end{equation}  
is the rank $3$ permutation character of $\tilde G$ acting on the set of 
$1$-spaces of $V$.
The following statement is well known, see e.g. formula (1) of \cite{FST}:

\begin{lem}\label{quad1}
For $\e = \pm$, the character $\pi^\e$ of the permutation action of $\tilde G$ on quadratic forms of type $\e$ associated to
$(\cdot,\cdot)$ is given as follows:
$$ \pi^+ = 1_{\tilde G} + \rho^2_n +  \sum^{(q-2)/2}_{i=1}\tau^i_n,~~~
  \pi^- = 1_{\tilde G} + \rho^1_n +  \sum^{(q-2)/2}_{i=1}\tau^i_n.$$ 
\end{lem}

Given any $g \in \GL(V)$, let 
$$d(x,g):= \dim_{\overline{\F}_q}\Ker(g-x \cdot 1_{V \otimes_{\F_q}}\overline{\F}_q)$$
for any $x \in \overline{\F}_q^\times$, and define the {\it support} of $g$ to be
\[
\supp(g) := \dim(V)-\max_{x \in \overline{\F}_q^\times}d(x,g).
\]
Set
$$d(g):= \dim(V)-\supp(g).$$

\begin{prop}\label{rat-sp2}
Let $\tilde G = \Sp_{2n}(q)$ with $n \geq 3$ and $2|q$, and let $g \in \tilde G$ have support $s=\supp(g)$. If 
$\chi \in \{\rho^1_n,\rho^2_n\}$, then 
$$\frac{|\c(g)|}{\c(1)} \leq \frac{1}{q^{s/3}}.$$
\end{prop}

\begin{proof}
The statement is obvious if $s=0$. Suppose $s=1$. It is easy to see that in this case $g$ is a transvection, and so 
$$\rho^1_n(g) = \rho^2_n(g) = \frac{q^{2n-1}-q}{2(q-1)}$$
by \cite[Corollary 7.8]{GT}, and the statement follows.

From now on we may assume $s \geq 2$. Observe that
$d:=\max_{x \in \F_q^\times}d(x,g) \leq d(g) = 2n-s$. Hence, 
$$0 \leq \rho(g) = \sum_{x \in \F_q^\times}\frac{q^{d(x)}-1}{q-1} \leq q^d-1,$$
and so \eqref{eq:dec11} implies 
$$|\rho^1_n(g)+\rho^2_n(g)| \leq q^d-1.$$
On the other hand, since $\pi^\pm(g) \geq 0$ and $\pi^++\pi^-$ is just the permutation character of $\tilde G$ acting on
$V$, Lemma \ref{quad1} implies that 
$$|\rho^1_n(g)-\rho^2_n(g)| = |\pi^+(g)-\pi^-(g)| \leq \pi^+(g)+\pi^-(g) = q^{d(1,g)} \leq q^d.$$ 
It follows for any $i \in \{1,2\}$ that
$$|\rho^i_n(g)| \leq \bigl(|\rho^1_n(g)+\rho^2_n(g)|+|\rho^1_n(g)+\rho^2_n(g)|\bigr)/2 < q^d \leq q^{2n-s}.$$
Since $n \geq 3$, we can also check that
$$\rho^i_n(1) \geq \frac{(q^n+1)(q^n-q)}{2(q-1)} > q^{2n-4/3}.$$
Thus $|\c(g)|/\c(1)| < q^{4/3-s} \leq q^{-s/3}$, as stated.
\end{proof}

Next we endow $V = \F_q^{2n}$ with a non-degenerate quadratic form $\QF$ of type $\e = \pm$ associated to 
the alternating form $(\cdot,\cdot)$. Choose a Witt basis $(e_1,\ldots,e_n,f_1, \ldots, f_n)$ for $(\cdot,\cdot)$, such that
$\QF(e_1)=\QF(f_1)=0$. We may assume that $P = \Stab_G(\langle e_1 \rangle_{\F_q}) = QL$, where 
$Q$ is elementary abelian of order $q^{2n-2}$, $L \cong \O^\e_{2n-2}(q) \times C_{q-1}$, and 
$$[P,P] = \Stab_G(e_1)=Q \rtimes [L,L]$$ 
has index $(q^n-\e)(q^{n-1}+\e)$ in $G$. Also consider $H := \Stab_G(e_1+f_1)$.

According to \cite[Theorem 1.3]{N}, $G$ has $q+1$ 
non-principal complex irreducible characters of degree at most $(q^n-\e)(q^{n-1}+\e)$, namely, $\a$ of degree
$(q^n-\e)(q^{n-1}+\e q)/(q^2-1)$, $\b$ of degree $(q^{2n}-q^2)/(q^2-1)$, 
$\g_i$ of degree $(q^n-\e)(q^{n-1}+\e)/(q-1)$, $1 \leq i \leq (q-2)/2$, and $\d_j$ of degree
$(q^n-\e)(q^{n-1}-\e)/(q+1)$, $1 \leq j \leq q/2$.

\begin{prop}\label{dec-so2}
Let $G = \O^\e_{2n}(q)$ with $n \geq 5$ and $2|q$, and consider $P = \Stab_G(e_1)$ and $H = \Stab_G(e_1+f_1)$ as
above. Then the following statements hold.
\begin{enumerate}[\rm(i)]
\item $\Ind^G_P(1_P) = 1_G + \a + \b$.
\item $\Ind^G_{[P,P]}(1_{[P,P]}) = 1_G +\a+\b + 2\sum^{(q-2)/2}_{i=1}\g_i$.
\item $\Ind^G_H(1_H) = 1_G +\b + \sum^{(q-2)/2}_{i=1}\g_i+\sum^{q/2}_{j=1}\d_j$.
\end{enumerate}

\end{prop}

\begin{proof}
(i) is well known. Next, $P/[P,P] \cong C_{q-1}$ has $q-1$ irreducible characters: $1_P$ and $(q-2)/2$ pairs of 
$\{\nu_i,\overline\nu_i\}$, $1 \leq i \leq (q-2)/2$. An application of Mackey's formula shows that 
$\Ind^G_P(\nu_i) = \Ind^G_P(\overline\nu_i)$ is irreducible for all $i$. Now using (i) we can write
\begin{equation}\label{eq:dec1}
  \Ind^G_{[P,P]}(1_{[P,P]}) = \Ind^G_P\bigl( \Ind^P_{[P,P]}(1_{[P,P]}) \bigr) =
    1_G+\a+\b + 2\sum^{(q-2)/2}_{i=1}\Ind^G_P(\nu_i).
\end{equation}    
On the other hand, note that $[P,P]$ has exactly $2q-1$ orbits on the set of nonzero singular vectors in $V$:
$q-1$ orbits $\{xe_1\}$ with $x \in \F_q^\times$, one orbit $\{v \in e_1^\perp \smallsetminus \langle e_1 \rangle_{\F_q} \mid \QF(v)=0\}$, and $(q-1)$ orbits 
$\{yf_1 + v \mid v \in e_1^\perp, \QF(yf_1+v) =0\}$ with $y \in \F_q^\times$. Together with \eqref{eq:dec1}, this implies 
that all summands in the last decomposition in \eqref{eq:dec1} are pairwise distinct. 
Since $\g_i = (q^n-\e)(q^{n-1}+\e)/(q-1) = \Ind^G_P(\nu_{i'})$, renumbering the $\nu_i$ if necessary, we may assume 
that $\Ind^G_P(\nu_i)=\g_i$, and (ii) follows.

\smallskip
For (iii), first note that $P$ has two orbits on the set $\XC := \{ v \in V \mid \QF(v)=1\}$, namely, $\XC \cap e_1^\perp$ and 
$\XC \smallsetminus e_1^\perp$. Since $\Ind^G_H(1_H)$ is the character of the permutation action of $G$ on $\XC$, we get
\begin{equation}\label{eq:dec2}
  [\Ind^G_P(1_P),\Ind^G_H(1_H)]_G = 2.
\end{equation}
Next, $[P,P]$ has  $q$ orbits on $\XC$, namely, $\XC \cap e_1^\perp$, and $\{yf_1+w \in \XC \mid w \in e_1^\perp\}$ with
$y \in \F_q^\times$. Thus
\begin{equation}\label{eq:dec3}
  [\Ind^G_{[P,P]}(1_{[P,P]}),\Ind^G_H(1_H)]_G = q.
\end{equation}  
Combining the results of (i), (ii), with \eqref{eq:dec2}, \eqref{eq:dec3}, and again using \cite[Theorem 1.3]{N}, we can write
\begin{equation}\label{eq:dec4}
  \Ind^G_H(1_H) = 1_G + (a\a + b\b) +\sum^{(q-2)/2}_{i=1}c_i\g_i + \sum^{q/2}_{j=1}d_j\d_j,
\end{equation}
where $a,b,c_i,d_j \in \Z_{\geq 0}$, $a+b=1$, $\sum_ic_i = (q-2)/2$.

\smallskip
Let $\tau$ denote the character of the permutation action of $G$ on $V \smallsetminus \{0\}$, so that
$$\tau = \Ind^G_{[P,P]}(1_{[P,P]}) + (q-1)\Ind^G_H(1_H).$$
Note that $G$ has $q^3+q^2-q$ orbits on $(V \smallsetminus \{0\}) \times (V \smallsetminus \{0\})$, namely,
$q(q-1)$ orbits of $(u,xu)$, where $x \in \F_q^\times$ and $\QF(u) = y \in \F_q$, and $q^3$ orbits of 
$(u,v)$, where $u,v$ are linearly independent and $(\QF(u),(u,v),\QF(v)) = (x,y,z) \in \F_q^3$. In other words,
$[\tau,\tau]_G = q^3+q^2-q$. Using (ii) and \eqref{eq:dec3}, we deduce that
\begin{equation}\label{eq:dec5}
  [\Ind^G_H(1_H),\Ind^G_H(1_H)]_G = q+1.
\end{equation}  
In particular, if $q=2$ then $\Ind^G_H(1_H)$ is the sum of $3$ pairwise distinct irreducible characters. By checking the 
degrees of $\a,\b$ and $\d_1$, (iii) immediately follows from \eqref{eq:dec4}.

\smallskip
Now we may assume $q=2^e \geq 4$. Let $\ell_+ = \ell(2^{ne}-1)$ denote a primitive prime divisor of $2^{ne}-1$,
which exists by \cite{Zs}. Likewise, let $\ell_- = \ell(2^{2ne}-1)$ denote a primitive prime divisor of $2^{2ne}-1$.
Then note that $\ell_\e$ divides the degree of each of $\a$, $\g_i$, $d_j$, but neither $[G:H]-1$ nor $\b(1)$. 
Hence \eqref{eq:dec4} implies that $(a,b)=(0,1)$. Comparing the degrees in \eqref{eq:dec4}, we also see that 
$\sum_jd_j = q/2$. Now 
$$q+1 = [\Ind^G_H(1_H),\Ind^G_H(1_H)]_G = 2 + \sum^{(q-2)/2}_{i=1}c_i^2 + \sum^{q/2}_{j=1}d_j^2 
    \geq 2 +   \sum^{(q-2)/2}_{i=1}c_i + \sum^{q/2}_{j=1}d_j = 2 +\frac{q-2}{2}+\frac{q}{2},$$
yielding $c_i^2=c_i$, $d_j^2=d_j$, $c_i,d_j \in \{0,1\}$, and so $c_i = d_j = 1$, as desired.     
\end{proof}

In the next statement, we embed $G = \O(V)$ in $\tilde G := \Sp(V)$ (the isometry group of the form $(\cdot,\cdot)$ on $V$).

\begin{prop}\label{sp-so1}
Let $n \geq 5$, $2|q$, and $\e = \pm$. Then the characters $\rho^1_n$ and $\rho^2_n$ of $\Sp(V) \cong \Sp_{2n}(q)$
restrict to $G = \O(V) \cong \O^\e_{2n}(q)$ as follows:
$$\begin{array}{ll}(\rho^1_n)|_{\O^+_{2n}(q)} = \b + \sum^{q/2}_{j=1}\d_j, &
    (\rho^2_n)|_{\O^+_{2n}(q)} = 1+\a+\b + \sum^{(q-2)/2}_{i=1}\g_i,\\ 
     (\rho^1_n)|_{\O^-_{2n}(q)} = 1+\a+\b + \sum^{(q-2)/2}_{i=1}\g_i, & (\rho^2_n)|_{\O^-_{2n}(q)} = \b + \sum^{q/2}_{j=1}\d_j.
     \end{array}$$
\end{prop}

\begin{proof}
Note by \eqref{eq:dec11} 
that $1_G + (\rho^1_n+\rho^2_n)|_G$ is just the character of the permutation action on the set of $1$-spaces of
$V$. Hence, by Proposition \ref{dec-so2} we have
\begin{equation}\label{eq:dec21}
 \bigl( \rho^1_n+\rho^2_n \bigr)|_G = \Ind^G_P(1_P) + \Ind^G_H(1_H) -1_G = 1_G +\a+2\b + \sum^{(q-2)/2}_{i=1}\g_i 
 + \sum^{q/2}_{j=1}\d_j. 
\end{equation}
Furthermore, Lemma \ref{quad1} implies by Frobenius' reciprocity that 
\begin{equation}\label{eq:dec22}
  \bigl(\rho^2_n\bigr)|_G \mbox { contains }1_G \mbox { when }\e=+, \mbox{ and }\bigl(\rho^1_n\bigr)|_G \mbox { contains }1_G \mbox { when }\e=-.
\end{equation}

\smallskip
(i) First we consider the case $\e = +$.  If $(n,q) \neq (6,2)$, one can find a primitive prime divisor $\ell = \ell(2^{ne}-1)$, where 
$q = 2^e$. If $(n,q) = (6,2)$, then set $\ell = 7$. By its choice, $\ell$ divides the degrees of $\rho^2_n$, $\a$, $\g_i$, and
$\d_j$, but $\b(1) \equiv \rho^1_n(1) \equiv -1 (\bmod\ \ell)$. Hence, \eqref{eq:dec21} and \eqref{eq:dec22} imply that
$$\bigl(\rho^2_n\bigr)|_G = 1_G +\b +x\a +  \sum^{(q-2)/2}_{i=1}y_i\g_i +  \sum^{q/2}_{j=1}z_j\d_j,$$
where $x,y_i,z_j \in \{0,1\}$. Setting $y:=\sum^{(q-2)/2}_{i=1}y_i$ and $z:=\sum^{q/2}_{j=1}z_j$ and comparing the degrees,
we get
$$(1-x)(q^{n-1}+q)+(q^{n-1}+1)(q+1)((q-2)/2-y) = z(q^{n-1}-1)(q-1),$$
and so $q^{n-1}+1$ divides $(1-x+2z)(q-1)$. Note that $\gcd(q-1,q^{n-1}+1)=1$ and 
$0 \leq (1-x+2z)(q-1) \leq q^2-1 < q^{n-1}+1$. It follows that $x=1$, $z=0$, $y=(q-2)/2$, whence $y_i=1$ and $z_j=1$, as stated.

\smallskip
(ii)  Now let $\e = -$, and choose $\ell$ to be a primitive prime divisor $\ell(2^{2ne}-1)$. By its choice, $\ell$ divides the degrees of $\rho^1_n$, $\a$, $\g_i$, and
$\d_j$, but $\b(1) \equiv \rho^2_n(1) \equiv -1 (\bmod\ \ell)$. Hence, \eqref{eq:dec21} and \eqref{eq:dec22} imply that
$$\bigl(\rho^1_n\bigr)|_G = 1_G +\b +x\a +  \sum^{(q-2)/2}_{i=1}y_i\g_i +  \sum^{q/2}_{j=1}z_j\d_j,$$
where $x,y_i,z_j \in \{0,1\}$. Setting $y:=\sum^{(q-2)/2}_{i=1}y_i$ and $z:=\sum^{q/2}_{j=1}z_j$ and comparing the degrees,
we get
$$(1-x)(q^{n-1}-q)+(q^{n-1}-1)(q+1)((q-2)/2-y) = z(q^{n-1}+1)(q-1),$$
and so $(q^{n-1}-1)/(q-1)$ divides $1-x+2z$. Since
$0 \leq 1-x+2z \leq q+1 < (q^{n-1}-1)/(q-1)$, it follows that $x=1$, $z=0$, $y=(q-2)/2$, whence $y_i=1$ and $z_j=1$, as stated.
\end{proof}

For the subsequent discussion, we recall the {\it quasi-determinant} $\kappa_\e: \GO_\e \to \{-1,1\}$,
where $\GO_\e:= \mathrm{GO}(V) \cong \mathrm{GO}^\e_{2n}(q)$, defined via
$$\kappa_\e(g) := (-1)^{\dim_{\F_q}\Ker(g-1_V)}.$$
It is known, see e.g. \cite[Lemma 5.8(i)]{GT}, that $\kappa$ is a group homomorphism, with
\begin{equation}\label{eq:kappa1}
  \Ker(\kappa_\e) = \O_\e:= \O(V) \cong \O^\e_{2n}(q).
\end{equation}   
Now we prove the ``unitary'' analogue of Lemma \ref{quad1}:

\begin{lem}\label{quad2}
For $n \geq 3$ and $2|q$, the following decompositions hold:
$$\Ind^{\tilde G}_{\GO_+}(\kappa_+) = \b_n +  \sum^{q/2}_{i=1}\zeta^i_n,~~~
    \Ind^{\tilde G}_{\GO_-}(\kappa_-) = \a_n +  \sum^{q/2}_{i=1}\zeta^i_n.$$ 
\end{lem}

\begin{proof}
According to formulae (10) and (4)--(6) of \cite{GT},
\begin{equation}\label{eq:dec31}
  \Ind^{\tilde G}_{\O_+}(\kappa_+)+ \Ind^{\tilde G}_{\O_-}(\kappa_-) = \a_n+\b_n +  2\sum^{q/2}_{i=1}\zeta^i_n.
\end{equation}  
Hence we can write
\begin{equation}\label{eq:dec32}
  \Ind^{\tilde G}_{\O_+}(\kappa_+) = x\a_n+y\b_n +  \sum^{q/2}_{i=1}z_i\zeta^i_n,
\end{equation}  
where $x,y,z_i \in \Z_{\geq 0}$, $x,y \leq 1$ and $z_i \leq 2$.
Note that, since $\pi^+= \Ind^{\tilde G}_{\GO_+}(1_{\GO_+})$, Lemma \ref{quad1} implies that 
$$|\GO_+ \backslash \tilde G/\GO_+| = \frac{q}{2}+1.$$
Next, by Mackey's formula we have
$$[\Ind^{\tilde G}_{\O_+}(\kappa_+),\Ind^{\tilde G}_{\O_+}(\kappa_+)]_G = \sum_{\GO_+t\GO_+ \in \GO_+ \backslash \tilde G/\GO_+}
    [(\kappa_+)|_{\GO_+ \cap t\GO_+t^{-1}},(\kappa^t_+)|_{\GO_+ \cap t\GO_+t^{-1}}]_{\GO_+ \cap t\GO_+t^{-1}},$$
where $\kappa^t_+(x) = \kappa(x^t) := \kappa(t^{-1}xt)$ for any $x \in \GO_+ \cap t\GO_+t^{-1}$. For such an $x$, note that
\begin{equation}\label{eq:dec321}
  \kappa_+(x) = 1 \Leftrightarrow 2 | \dim_{\F_q}\Ker(x-1_V) \Leftrightarrow 2 | \dim_{\F_q}\Ker(x^{t-1}-1_V) \Leftrightarrow 
    (\kappa_+)^t(x) = 1,
\end{equation}    
i.e. $\kappa_+(x) = \kappa^t_+(x)$. It follows that
\begin{equation}\label{eq:dec33}
  x^2+y^2+\sum^{q/2}_{i=1}z_i^2=
    [\Ind^{\tilde G}_{\O_+}(\kappa_+),\Ind^{\tilde G}_{\O_+}(\kappa_+)]_G = |\GO_+ \backslash \tilde G/\GO_+| = \frac{q}{2}+1.
\end{equation}    
On the other hand, equating the character degrees in \eqref{eq:dec32} we obtain 
\begin{equation}\label{eq:dec34}
  \frac{q^n(q^n+1)}{2} =  x\frac{(q^n-1)(q^n-q)}{2(q+1)}+y\frac{(q^n+1)(q^n+q)}{2(q+1)}+\sum^{q/2}_{i=1}z_i \cdot 
    \frac{q^{2n}-1}{q+1}.
\end{equation}
We claim that $x=0$. Indeed, if $(n,q) = (3,2)$, then \eqref{eq:dec34} implies that $3|x$, and so $x=0$ as $0 \leq x \leq 1$.
Assume $(n,q) \neq (3,2)$. Then we can find a primitive prime divisor $\ell = \ell(2^{2ne}-1)$ for $q = 2^e$, and note
from \eqref{eq:dec34} that $\ell|x$. Since $\ell > 2$ and $x \in \{0,1\}$, we again have $x=0$.

Now if $y=0$, then \eqref{eq:dec34} implies that $q^n(q^n+1)/2$ is divisible by $(q^{2n}-1)/(q+1)$, a contradiction. Hence
$y=1$, and from \eqref{eq:dec34} we obtain that $\sum^{q/2}_{i=1}z_i = q/2$. On the other hand, 
$\sum^{q/2}_{i=1}z_i^2 = q/2$ by \eqref{eq:dec33}. Thus $\sum^{q/2}_{i=1}(z_i-1)^2 = 0$, and so $z_i = 1$ for all
$i$. Together with \eqref{eq:dec31}, this yields the two stated decompositions.          
\end{proof}

\begin{prop}\label{sp-so2}
Let $n \geq 5$, $2|q$, and $\e = \pm$. Then the characters $\a_n$ and $\b_n$ of $\Sp(V) \cong \Sp_{2n}(q)$
restrict to $G = \O(V) \cong \O^\e_{2n}(q)$ as follows:
$$\begin{array}{ll}(\a_n)|_{\O^+_{2n}(q)} = \sum^{q/2}_{j=1}\d_j, &
    (\b_n)|_{\O^+_{2n}(q)} = 1+\a + \sum^{(q-2)/2}_{i=1}\g_i,\\ 
     (\a_n)|_{\O^-_{2n}(q)} = 1+\a + \sum^{(q-2)/2}_{i=1}\g_i, & (\b_n)|_{\O^-_{2n}(q)} = \sum^{q/2}_{j=1}\d_j.
     \end{array}$$
In particular, the following formula holds for the irreducible character $\b$ of $G$ of degree $(q^{2n}-q^2)/(q^2-1)$:
$$\bigl( (\rho^1_n+\rho^2_n)-(\a_n+\b_n)\bigr)|_{\O^\e_{2n}(q)} = 2\b.$$  
\end{prop}

\begin{proof}
By Mackey's formula,
$$\bigl(\Ind^{\tilde G}_{\O_+}(\kappa_+)\bigr)|_G = \sum_{Gt\GO_+\in G \backslash \tilde G/\GO_+}
    \Ind^G_{G \cap t\GO_+t^{-1}}\bigl((\kappa^t_+)|_{G \cap t\GO_+t^{-1}}\bigr),$$
and similarly for $\pi^+=\Ind^{\tilde G}_{\O_+}(1_{\GO_+})$.
The argument in \eqref{eq:dec321} shows that $\kappa^t_+(x)=1$ for all $x \in G \cap t\GO_+t^{-1}$, and so 
$\pi^+$ and $\Ind^{\tilde G}_{\O_+}(\kappa_+)$ agree on $G$. Similarly, 
$\pi^-$ and $\Ind^{\tilde G}_{\O_-}(\kappa_-)$ agree on $G$. It then follows from Lemmas \ref{quad1} and 
\ref{quad2} that
\begin{equation}\label{eq:dec41}
  \bigl( \rho^2_n-\rho^1_n\bigr)|_G = \bigl( \pi^+-\pi^-\bigr)|_G = 
    \bigl( \Ind^{\tilde G}_{\O_+}(\kappa_+)- \Ind^{\tilde G}_{\O_-}(\kappa_-) \bigr)|_G = \bigl( \b_n-\a_n\bigr)|_G.
\end{equation}    
First assume that $\e=+$. Then using Proposition \ref{sp-so1} we get
$$\bigl( \b_n-\a_n\bigr)|_G= 1_G+\a+\sum^{(q-2)/2}_{i=1}\g_i-\sum^{q/2}_{j=1}\d_j,$$
i.e. 
$$\sum^{q/2}_{j=1}\d_j+(\b_n)|_G = 1_G +\a+\sum^{(q-2)/2}_{i=1}\g_i + (\a_n)|_G.$$  
Aside from $(\a_n)|_G$ and $(\b_n)|_G$, all the other characters in the above equality are irreducible and 
pairwise distinct. It follows that $(\a_n)|_G$ contains $\sum^{q/2}_{j=1}\d_j$. Comparing the degrees, we 
see that 
$$(\a_n)|_G = \sum^{q/2}_{j=1}\d_j,$$ 
which then implies that  
$$(\b_n)|_G = 1_G +\a+\sum^{(q-2)/2}_{i=1}\g_i.$$ 
Now assume that $\e=-$. Then again using Proposition \ref{sp-so1} and \eqref{eq:dec41} we get
$$\bigl( \a_n-\b_n\bigr)|_G= 1_G+\a+\sum^{(q-2)/2}_{i=1}\g_i-\sum^{q/2}_{j=1}\d_j,$$
i.e. 
$$\sum^{q/2}_{j=1}\d_j+(\a_n)|_G = 1_G +\a+\sum^{(q-2)/2}_{i=1}\g_i + (\b_n)|_G.$$  
Aside from $(\a_n)|_G$ and $(\b_n)|_G$, all the other characters in the above equality are irreducible and 
pairwise distinct. It follows that $(\b_n)|_G$ contains $\sum^{q/2}_{j=1}\d_j$. Comparing the degrees, we 
see that 
$$(\b_n)|_G = \sum^{q/2}_{j=1}\d_j,$$ 
which then implies that  
$$(\a_n)|_G = 1_G +\a+\sum^{(q-2)/2}_{i=1}\g_i.$$ 
For both $\e = \pm$, the last statement now follows from \eqref{eq:dec21}.
\end{proof}

Proposition \ref{sp-so2} leads to the following explicit
formula for $\b$, which we will show to hold for all special orthogonal groups in all characteristics and all dimensions, and which
is of independent interest.
In this result, we let $V = \F_q^n$ be a quadratic space, $L := \SO(V)$ if $2 \nmid q$, 
$L := \O(V)$ if $2|q$, and extend the action of $L$ on $V$ to $\tilde V := V \otimes_{\F_q}\F_{q^2}$, and we assume
$2 \nmid q$ if $2 \nmid n$. Also, set
$$\mu_{q-1}:= \F_q^\times,~~\mu_{q+1} := \{ x \in \F_{q^2}^\times \mid x^{q+1} = 1 \}.$$
If $2 \nmid q$, let $\chi_2^+$ be the unique linear character of order $2$ of $\mu_{q-1}$, and let $\chi_2^-$ be the unique linear character of order $2$ of $\mu_{q+1}$.

\begin{thm}\label{beta-so2}
Let $n \geq 10$, $\e = \pm$, and let $q$ be any prime power. If $2|n$, let $\psi = \b$ be the irreducible constituent $\b$ of 
degree $(q^{n}-q^2)/(q^2-1)$ of the rank $3$ permutation character of $L = \O(V)$ when $2|q$, and 
of $L = \SO(V)$ when $2 \nmid q$, on the set of singular 
$1$-spaces of its natural module $V=\F_q^{n}$. If $2 \nmid qn$, let $\psi$ be the irreducible character of 
$L = \SO(V)$ of degree $(q^n-q)/(q^2-1)$ denoted by $D_{\St}$ in \cite[Proposition 5.7]{LBST}. Then for any $g \in L$ we have 
$$\psi(g) = \frac{1}{2(q-1)}\sum_{\l \in \mu_{q-1}}q^{\dim_{\F_q}\Ker(g - \l \cdot 1_V)} - 
    \frac{1}{2(q+1)}\sum_{\l \in \mu_{q+1}}(-q)^{\dim_{\F_{q^2}}\Ker(g - \l \cdot 1_{\tilde V})}  
    -1$$
when $2|n$, and by 
$$\psi(g) = \frac{1}{2(q-1)}\sum_{\l \in \mu_{q-1}}\chi_2^+(\l)q^{\dim_{\F_q}\Ker(g - \l \cdot 1_V)} +
    \frac{1}{2(q+1)}\sum_{\l \in \mu_{q+1}}\chi^-_2(\l)(-q)^{\dim_{\F_{q^2}}\Ker(g - \l \cdot 1_{\tilde V})}  
 $$ 
when $2 \nmid qn$.   
\end{thm}

\begin{proof}
In the case $2|q$, the statement follows from the last formula in Proposition \ref{sp-so2}, together with formulae (3) and (6) of \cite{GT}. Assume now that $2 \nmid q$, and set $\kappa := 1$ if $2|n$ and $\kappa := 0$ if $2 \nmid n$.
By \cite[Proposition 5.7]{LBST} (and in the notation of \cite[\S5.1]{LBST}), 
$$\psi(g)=\frac{1}{|\Sp_2(q)|}\sum_{x \in \Sp_2(q)}\omega_{n}(xg)\St(x)-\kappa,$$
where $\omega_{2n}$ denotes a reducible Weil character of $\Sp_{2n}(q)$ and $\St$ denotes the Steinberg character of
$S:= \Sp_2(q)$. 

If $x \in S$ is not semisimple, then $\St(x) = 0$.

Suppose $x = \mathrm{diag}(\l,\l^{-1}) \in T_1 <S$, where $T_1 \cong C_{q-1}$ is a split torus and $\l \in \mu_{q-1}$.
In this case, we can view $T_1$ as $\GL_1(q)$, embed $G$ in $\GL_{n}(q)$, and view $xg$ as 
an element $h=\l g$ in a Levi subgroup $\GL_{n}(q)$ of $\Sp_{2n}(q)$, with $\det(h) = \l^{n}$. It follows from 
\cite[Theorem 2.4(c)]{Ge} that 
$$\omega_{n}(xg) = \chi_2^+(\l^n) q^{\dim_{\F_q}\Ker(h-1)} = \chi_2^+(\l^n) q^{\dim_{\F_q}\Ker(g-\l^{-1})}.$$
If $\l \neq \pm 1$, then $|x^S| = q(q+1)$ and $\St(x) = 1$. If $\l = \pm 1$, then $|x^S| = 1$ and $\St(x)=q$.
Note that since $g \in \mathrm{GO}(V)$, 
$$\dim_{\F_q}\Ker(g-\l^{-1}) = \dim_{\F_q}\Ker(\tw t g-\l^{-1}) = \dim_{\F_q}\Ker(g^{-1}-\l^{-1}) = \dim_{\F_q}\Ker(g-\l).$$
We also note that since $g \in \SO(V)$, 
\begin{equation}\label{eq:kappa2}
  \dim_{\F_q}\Ker(g-1_V) \equiv n (\mod 2),~~ \dim_{\F_q}\Ker(g+1_V) \equiv 0 (\mod 2).
\end{equation}
(Indeed, since $\det(g)=1$, each of $\Ker(g_s-1_V)$ and $\Ker(g_s+1_V)$ is a non-degenerate subspace of $V$ if nonzero, where $g=g_sg_u$ is the Jordan decomposition; furthermore, $2|\dim_{\F_q}\Ker(g_s+1_V)$ and 
$\dim\Ker_{\F_q}\Ker(g_s-1_V) \equiv n (\mod 2)$. Hence the claim reduces to the unipotent case $g=g_u$. In the latter case,
the number of Jordan blocks of $g_u$ of each even size is even, see \cite[\S13.1]{Car}, and the claim follows.) 

Suppose $x = \mathrm{diag}(\mu,\mu^{-1}) \in T_2 <S$, where $T_2 \cong C_{q+1}$ is a non-split torus and 
$\mu \in \mu_{q+1}$ with $\mu \neq \pm 1$. Then $\St(x) = -1$ and $|x^S| = q(q-1)$. In this case, we can view $T_2$ as 
$\GU_1(q)$, embed $G$ in $\GU_{n}(q)$, and view $xg$ as 
an element $h=\mu g$ in a subgroup $\GU_{n}(q)$ of $\Sp_{2n}(q)$, with $\det(h) = \mu^{n}$. It follows from
\cite[Theorem 3.3]{Ge} that 
$$\omega_{n}(xg) = (-1)^n\chi_2^-(\mu^n)(-q)^{\dim_{\F_{q^2}}\Ker(h-1)} = (-1)^n\chi_2^-(\mu^n)(-q)^{\dim_{\F_{q^2}}\Ker(g-\mu^{-1})}.$$
Altogether, we have shown that
\begin{equation}\label{eq:dec51}
  \begin{aligned}\psi(g) & = \frac{1}{q^2-1}\bigl(q^{\dim_{\F_q}\Ker(g-1)}+\chi^+_2((-1)^n)q^{\dim_{\F_q}\Ker(g+1)}\bigr)\\ 
    & +\frac{1}{2(q-1)}\sum_{\l \in \mu_{q-1} \smallsetminus \{\pm 1\}}\chi^+_2(\l^n)q^{\dim_{\F_q}\Ker(g -\l)}\\
     & - \frac{(-1)^n}{2(q+1)}\sum_{\mu \in \mu_{q+1} \smallsetminus \{\pm 1\}}\chi^-_2(\mu^n)(-q)^{\dim_{\F_{q^2}}\Ker(g-\mu)}
         -\kappa,\end{aligned}
\end{equation}    
and the statement now follows if we use \eqref{eq:kappa2}.
\end{proof}

\subsection{Some character estimates}

\begin{prop}\label{rat-so21}
Let $q$ be any prime power, $G = \O^\e_{2n}(q)$ with $n \geq 5$, $\e=\pm$, 
and let $g \in G$ have support $s=\supp(g)$. 
Assume that $\chi \in \{\a,\b\}$ if $2 \nmid q$, and $\chi \in \{\a,\b,\g_i\}$ if $2|q$. Then 
$$\frac{|\c(g)|}{\c(1)} \leq \frac{1}{q^{s/3}}.$$
\end{prop}

\begin{proof}
(i) First we consider the case $s \geq n \geq 5$. Then 
\begin{equation}\label{eq:rb1}
  d(x,g) \leq 2n-s 
\end{equation}  
for any $x \in \overline{\F}_q^\times$. In particular,
\begin{equation}\label{eq:rb2}
  0 \leq \rho(g) \leq \sum_{x \in \F_q^\times}\frac{q^{d(x,g)}-1}{q-1} \leq q^{2n-s}-1.
\end{equation}  
Now, (when $2|q$) part (i) of the proof of Proposition \ref{dec-so2} shows that 
$\g_i = \Ind^G_P(\nu_j)$ for some linear character $\nu_j$ of
$P$, and recall that $\rho = \Ind^G_P(1_P)$. It follows that 
$$|\g_i(g)| \leq |\rho(g)| \leq q^{2n-s}-1,$$
and so $|\g_i(g)/\g_i(1)| < 1/q^{s-2} \leq q^{-3s/5}$ as $\g_i(1) = [G:P] > q^{2n-2}$. Next, using Theorem \ref{beta-so2} and 
\eqref{eq:rb1} we also see that
\begin{equation}\label{eq:rb3}
  |\b(g)+1| \leq \frac{1}{2(q-1)}\sum_{x \in \F_q^\times}q^{d(x,g)}
    + \frac{1}{2(q+1)}\sum_{x \in \overline{\F}_q^\times,x^{q+1}=1}q^{d(x,g)} \leq q^{2n-s}.
\end{equation}    
In particular, $|\b(g)| \leq q^{2n-s}+1$. Since $\b(1) = (q^{2n}-q^2)/(q^2-1)$, we deduce that
$|\b(g)/\b(1)| < q^{-3s/5}$. Furthermore, as $\a(g) = \rho(g)-(\b(g)+1)$, we obtain from \eqref{eq:rb2}--\eqref{eq:rb3} that
$$|\a(g)| \leq 2q^{2n-s}-1.$$
If $s \geq 6$, then it follows that $|\a(g)/\a(1)| < q^{4-s} \leq q^{-s/3}$, since $\a(1) > q^{2n-3}$. Suppose that $s=n=5$. 
Then we can strengthen \eqref{eq:rb3} to 
$$\frac{-2q^5-(q-1)q^3}{2(q+1)} \leq \b(g)+1 \leq q^5.$$
Together with \eqref{eq:rb2}, this implies that
$$|\a(g)| = |\rho(g)-(\b(g)+1)| < q^5+q^4 < \a(1)/q^{s/3}$$
since $\a(1) \geq (q^5+1)(q^4-q)/(q^2-1)$.

\smallskip
(ii) From now on we may assume that $s \leq n-1$. As $g \in G=\O^\e_{2n}(q)$, it follows that $d(z,g) = 2n-s$ for a unique
$z \in \{1,-1\}$. Furthermore, $2|s$. (Indeed, this has been recorded in \eqref{eq:kappa1} when $2|q$, and 
in \eqref{eq:kappa2} when $2 \nmid q$.) We also have that
\begin{equation}\label{eq:rb4}
  d(x,g) \leq 2n-d(z,g) =s
\end{equation}  
for all $x \in \overline\F_q^\times \smallsetminus \{z\}$, 

Assume in addition that $s \geq 4$. Using \eqref{eq:rb4}
we obtain
\begin{equation}\label{eq:rb5}
  0 \leq \rho(g) \leq \frac{q^{2n-s}-1+(q-2)(q^s-1)}{q-1}.
\end{equation}  
As $\rho(1)=(q^n-\e)(q^{n-1}+\e)/(q-1)$, it follows that $|\rho(g)/\rho(1)| < q^{-3s/5}$. As above, the same bound also applies
to $\chi=\g_i$ when $2|q$. 

Next, since $2|s$, using Theorem \ref{beta-so2} and applying \eqref{eq:rb4} to $x^{q \pm 1} = 1$ and $x \neq z$, we have that
\begin{equation}\label{eq:rb6}
  \frac{q^{2n-s}}{q^2-1}-q^s \cdot \frac{q}{2(q+1)} \leq \beta(g)+1 \leq \frac{q^{2n-s}}{q^2-1}+ q^s \cdot \biggl( \frac{q-2}{2(q-1)} + 
  \frac{q}{2(q+1)} \biggr);
\end{equation}  
in particular, 
$$|\beta(g)| < \frac{q^{2n-s}+q^s(q^2-q-1)}{q^2-1}.$$
Since $\b(1) = (q^{2n}-q^2)/(q^2-1)$, we obtain that $|\b(g)/\b(1)| < q^{-4s/5}$. Furthermore, using 
\eqref{eq:rb5}--\eqref{eq:rb6}, we can bound 
$$|\a(g)| = |\rho(g)-(\b(g)+1)| < \frac{q^{2n-s+1}+q^s(3q^2-3q-4)/2}{q^2-1} <\frac{\a(1)}{q^{2s/5}}$$
since $\a(1) \geq (q^n+1)(q^{n-1}-q)/(q^2-1)$.

\smallskip
(iii) Since the statement is obvious for $s=0$, it remains to consider the case $s=2$, i.e. $d(1,zg) = 2n-2$.
Using \cite[Lemma 4.9]{TZ1}, one can readily show that $g$ fixes an orthogonal decomposition $V = U \oplus U^\perp$, with 
$U \subset \Ker(g-z \cdot 1_V)$ being non-degenerate of dimension $2n-4$, and 
\begin{equation}\label{eq:rb7}
  \dim_{\F_q}(U^\perp)^{zg} = 2.
\end{equation}  

First we estimate $\rho(g)$. Suppose $g(v) = tv$ for some singular $0 \neq v \in V$ and $t \in \F_q^\times$. If $t \neq z$,
then $v \in U^\perp$, and \eqref{eq:rb7} implies that $g$ can fixes at most $q+1$ such singular $1$-spaces 
$\langle v \rangle_{\F_q}$. Likewise, $g$ fixes at most $q+1$ singular $1$-spaces $\langle v \rangle_{\F_q} \subset U^\perp$
with $g(v) = zv$. Assume now that $g(v) = zv$ with $v = u+u'$, $0 \neq u \in U$ and $u' \in U^\perp$. 
As $0 = \QF(v) = \QF(u)+\QF(u')$, the total number of such $v$ is 
$$N:=\sum_{x \in \F_q}|\{ 0 \neq w \in U \mid \QF(w) = x \}| \cdot |\{ w' \in U^\perp \mid g(w') = zw',\QF(w') = -x \}|.$$ 
Note that, since $U$ is a non-degenerate quadratic space of dimension $2n-4$, 
$$(q^{n-2}+1)(q^{n-3}-1) \leq |\{ 0 \neq w \in U \mid \QF(w) = x \}| \leq  (q^{n-2}-1)(q^{n-3}+1)$$
for any $x \in \F_q$. On the other hand, \eqref{eq:rb7} implies that
$$\sum_{x \in \F_q}|\{ w' \in U^\perp \mid g(w') = zw',\QF(w') = -x \}| = |(U^\perp)^{zg}| = q^2.$$
It follows that  
$$q^2(q^{n-2}+1)(q^{n-3}-1) \leq N \leq  q^2(q^{n-2}-1)(q^{n-3}+1),$$
and so
\begin{equation}\label{eq:rb8}
   \frac{q^2(q^{n-2}+1)(q^{n-3}-1)}{q-1} \leq \rho(g) \leq 2q+2+\frac{q^2(q^{n-2}-1)(q^{n-3}+1)}{q-1}.
\end{equation}  
In particular, when $2|q$ we have $|\g_i(g)| \leq |\rho(g)| < \rho(1)/q^{4s/5}$.

Next, applying \eqref{eq:rb6} to $s=2$ we have 
$$|\b(g)| \leq \frac{q^{2n-2}+q^2(q^2-q-1)}{q^2-1} < \frac{\b(1)}{q^{4s/5}}.$$
Finally, using \eqref{eq:rb6} with $s=2$ and \eqref{eq:rb8}, we obtain 
$$|\a(g)| = |\rho(g)-(\b(g)+1)| < \frac{q^{2n-3}+q^{n+1}-q^{n-1}}{q^2-1}+(q+1) <\frac{\a(1)}{q^{3s/5}}.$$
\end{proof}

\begin{prop}\label{rat-sp-so22}
Let $q$ be any odd prime power, $n \geq 5$, and $\e=\pm$. 
Assume that $\chi \in \Irr(G)$, where either $G \in \{\Sp_{2n}(q), \O_{2n+1}(q)\}$ and $\chi \in \{\a,\b\}$, or 
$G = \O^\e_{2n}(q)$ and $\chi \in \{\a,\b,\g_i\}$. If $g \in G$ has support $s=\supp(g)$, then 
$$\frac{|\c(g)|}{\c(1)} \leq \frac{1}{q^{s/3}}.$$
\end{prop}

\begin{proof}
(i) As usual, we may assume $s \geq 1$. 
First we consider the case $G = \O^\e_{2n}(q)$. Then \cite[Corollary 5.14]{NT} and \cite[Proposition 5.7]{LBST} show 
(in their notation) that $\a=D_{1}-1_G$, $\b = D_{\St}-1_G$. Furthermore if $\nu \neq 1_P$ is a linear character of $P$, then 
$\Ind^G_P(\nu) = D_{\chi_j}$ if $\nu$ has order $>2$, and $\Ind^G_P(\nu) = D_{\xi_1}+D_{\xi_2}$ if $\nu$ has order $2$.

If $\chi = \a$ or $\b$, then the statement is already proved in Proposition \ref{rat-so21}, whose proof also applies to
the case $\chi=\g_i = D_{\chi_j}$ (using the estimate $|\Ind^G_P(\nu)(g)| \leq \rho(g)$). 
It remains to consider the case $\chi = \g_i = D_{\xi_j}$ for $j = 1,2$. Again the previous argument applied to 
$\nu$ of order $2$ shows that 
$$|D_{\xi_1}(g)+D_{\xi_2}(g)| \leq \frac{[G:P]}{q^{3s/5}} = \frac{2\chi(1)}{q^{3s/5}}.$$ 
On the other hand, the formula for $D_\a$ in \cite[Lemma 5.5]{LBST}, the character table of $\SL_2(q)$ \cite[Theorem 38.1]{D},
and part 1) of the proof of \cite[Proposition 5.11]{LBST} imply that
\begin{equation}\label{eq:rb21}
  |D_{\xi_1}(g)-D_{\xi_2}(g)| \leq \frac{2(q^2-1)q^n \cdot\sqrt{q}}{q(q^2-1)} = 2q^{n-1/2}.
\end{equation}  
If $4 \leq s \leq 2n-2$, then since $\chi(1) \geq (q^n+1)(q^{n-1}-1)/2(q-1)> q^{2n-3}(q+1)$ it follows that
$$\begin{aligned}|\chi(g)| & \leq \bigl(|D_{\xi_1}(g)+D_{\xi_2}(g)|+|D_{\xi_1}(g)+D_{\xi_2}(g)|\bigr)/2 \\
     & \leq \frac{\chi(1)}{q^{3s/ 5}}+q^{n-1/2}
   < \frac{\chi(1)}{q^{3s/5}} + \frac{2\chi(1)}{q^{s/3-1/6}(q+1)} < \frac{\chi(1)}{q^{s/3}}.\end{aligned}$$
If $1 \leq s \leq 4$, then $s < n$, and so $2|s$ as shown in part (ii) of the proof of Proposition \ref{rat-so21}. Hence $s=2$, and 
we again have
$$|\chi(g)|  \leq \frac{\chi(1)}{q^{3s/ 5}}+q^{n-1/2} < \frac{\chi(1)}{q^{3s/5}} + \frac{2\chi(1)}{q^{s/3+17/6}} < 
    \frac{\chi(1)}{q^{s/3}}.$$
Finally, if $s=2n-1$, then $d(x,g) \leq 1$ for all $x \in \overline\F_q^\times$ by \eqref{eq:rb1}; moreover, 
$d(\pm 1,g) = 0$. Hence, instead of \eqref{eq:rb21} we now have the stronger bound
$$|D_{\xi_1}(g)-D_{\xi_2}(g)| \leq \frac{2(q^2-1) \cdot\sqrt{q}}{q(q^2-1)} = 2q^{-1/2},$$
whence $|\chi(g)| \leq \chi(1)q^{-3s/5}+q^{-1/2} < \chi(1)q^{-s/3}$.  

\smallskip
(ii) Next we consider the case $G = \O^\e_{2n+1}(q)$. Then \cite[Corollary 5.15]{NT} and \cite[Proposition 5.7]{LBST} show 
(in their notation) that $\a=D_{\xi_1}-1_G$, $\b = D_{\xi_2}-1_G$. Again using the formula for $D_\a$ in \cite[Lemma 5.5]{LBST}, the character table of $\SL_2(q)$ \cite[Theorem 38.1]{D},
and part 1) of the proof of \cite[Proposition 5.11]{LBST}, we obtain that
\begin{equation}\label{eq:rb22}
  |\a(g)-\b(g)| = |D_{\xi_1}(g)-D_{\xi_2}(g)| \leq \frac{2(q^2-1)q^{n+1/2} \cdot\sqrt{q}}{q(q^2-1)} = 2q^n.
\end{equation}  

Suppose in addition that $3 \leq s \leq 2n-2$. Since $d(x,g) \leq 2n+1-s$ by \eqref{eq:rb1}, we have that
$$0 \leq \rho(g)=1+\a(g)+\b(g) \leq \sum_{x \in \mu_{q-1}}\frac{q^{d(x,g)}-1}{q-1} \leq q^{2n+1-s}.$$ 
As $\chi(1) \geq (q^n+1)(q^n-q)/2(q-1)$, it follows that
$$|\a(g)+\b(g)| \leq q^{2n+1-s}-1  < \frac{2(1-1/q)q^{2-s}\chi(1)}{(1+1/q^n)(1-1/q^{n-1})} < 
    \frac{2(1-1/q)\chi(1)}{q^{s/3}(1-1/q^{n-1})}.$$
On the other hand, \eqref{eq:rb22} implies that
$$|\a(g)-\b(g)| \leq \frac{4(1-1/q)\chi(1)}{q^{(s+4)/3}(1-1/q^{n-1})},$$
and so
$$\frac{|\chi(g)|}{\chi(1)} < \frac{(1-1/q)}{q^{s/3}(1-1/q^{n-1})}+ \frac{2(1-1/q)}{q^{(s+4)/3}(1-1/q^{n-1})} < \frac{1}{q^{s/3}}.$$
If $s=2n-1$ or $2n$, then $d(x,g) \leq 2$ for all $x \in \overline\F_q^\times$ by \eqref{eq:rb1}. 
Hence, instead of \eqref{eq:rb21} we now have the stronger bound
$$|\a(g)-\b(g)|=|D_{\xi_1}(g)-D_{\xi_2}(g)| \leq \frac{2(q^2-1)q^2 \cdot\sqrt{q}}{q(q^2-1)} = 2q^{3/2},$$
whence 
$$|\chi(g)| < \frac{(1-1/q)q^{2-s}\chi(1)}{(1-1/q^{n-1})} +q^{3/2} < \chi(1)q^{-s/3}.$$   

It remains to consider the case $s=1,2$, i.e. $d(1,zg) = 2n$ or $2n-1$ for some $z \in \{1,-1\}$.
Using \cite[Lemma 4.9]{TZ1}, one can readily show that $g$ fixes an orthogonal decomposition $V = U \oplus U^\perp$, with 
$U \subset \Ker(g-z \cdot 1_V)$ being non-degenerate of dimension $2n-3$, and 
\begin{equation}\label{eq:rb23}
  \dim_{\F_q}(U^\perp)^{zg} = 4-s.
\end{equation}  
First we estimate $\rho(g)$. Suppose $g(v) = tv$ for some singular $0 \neq v \in V$ and $t \in \F_q^\times$. If $t \neq z$,
then $v \in U^\perp$, and \eqref{eq:rb23} implies that $g$ can fixes at most $(q^s-1)/(q-1) \leq q+1$ such singular $1$-spaces 
$\langle v \rangle_{\F_q}$. Likewise, $g$ fixes at most $(q+1)^2$ singular $1$-spaces 
$\langle v \rangle_{\F_q} \subset U^\perp$ with $g(v) = zv$, since $\dim U^\perp = 4$. Assume now that $g(v) = zv$ with $v = u+u'$, $0 \neq u \in U$ and $u' \in U^\perp$. As $0 = \QF(v) = \QF(u)+\QF(u')$, the total number of such $v$ is 
$$N:=\sum_{x \in \F_q}|\{ 0 \neq w \in U \mid \QF(w) = x \}| \cdot |\{ w' \in U^\perp \mid g(w') = zw',\QF(w') = -x \}|.$$ 
Note that, since $U$ is a non-degenerate quadratic space of dimension $2n-3$, 
$$q^{n-2}(q^{n-2}-1) \leq |\{ 0 \neq w \in U \mid \QF(w) = x \}| \leq q^{n-2}(q^{n-2}+1)$$
for any $x \in \F_q$. On the other hand, \eqref{eq:rb23} implies that
$$\sum_{x \in \F_q}|\{ w' \in U^\perp \mid g(w') = zw',\QF(w') = -x \}| = |(U^\perp)^{zg}| = q^{4-s}.$$
It follows that  
$$q^{n+2-s}(q^{n-2}-1) \leq N \leq  q^{n+2-s}(q^{n-2}+1),$$
and so
$$\frac{q^{n+2-s}(q^{n-2}-1)}{q-1} \leq \rho(g)=1+\a(g)+\b(g) \leq q^2+3q+2+\frac{q^{n+2-s}(q^{n-2}+1)}{q-1}.$$
Together with \eqref{eq:rb22}, this implies that
$$\frac{|\chi(g)|}{\chi(1)} \leq \frac{(q^2-1)(q+2)+q^{n+2-s}(q^{n-2}+1)+2q^n(q-1)}{(q^n+1)(q^n-q)} < \frac{1}{q^{s/2}}.$$ 

\smallskip
(iii) Finally, we consider the case $G = \Sp_{2n}(q)$. In this case, arguing similarly to the proof of \cite[Proposition 5.7]{LBST},
one can show that $\{\a,\b\} = \{D^\circ_{\l_0},D^\circ_{\l_1}\}$, 
where $S = \GO^+_2(q) \cong D_{2(q-1)}$, with $\l_0$, $\l_1$ being
the two linear characters trivial at $\SO^+_2(q)$, and we consider the dual pairs $G \times S \to \Sp_{4n}(q)$. 
In particular, $\chi(1) \geq (q^n+1)(q^n-q)/2(q-1) > q^{2n-4/3}$.
Now, the formula for $D_\a$ in \cite[Lemma 5.5]{LBST}, the character table of $S$,
and part 1) of the proof of \cite[Proposition 5.11]{LBST} imply that
\begin{equation}\label{eq:rb24}
  |\a(g)-\b(g)| \leq q^{(d(1,g)+d(-1,g))/2} \leq q^{2n-s}.
\end{equation}  
On the other hand, using \eqref{eq:rb1} we have $0 \leq \rho(g) = \a(g)+\b(g)+1 \leq q^{2n-s}-1$. In particular, when 
$s \geq 2$ we have 
$$|\chi(g)| \leq \bigl(|\a(g)+\b(g)|+|\a(g)-\b(g)|\bigr)/2 \leq q^{2n-s} < \chi(1)q^{-s/3}.$$
Assume now that $s=1$. Then $g = zu$ for some $z = \pm 1$ and unipotent $u \in G$; furthermore,
$\rho(g) = (q^{2n-1}-1)/(q-1)$. Applying also \eqref{eq:rb24}, we obtain
$$|\chi(g)| \leq \biggl(|\a(g)+\b(g)|+|\a(g)-\b(g)|\biggr)/2 \leq \biggl(\frac{q^{2n-1}-q}{q-1}+q^{n-1/2}\biggr)/2 < \chi(1)q^{-4s/5},$$
and the proof is complete.
\end{proof}

\section{Classical groups: Proof of Theorem \ref{main1}}\label{pfth1}

Let $G = \Sp(V)$ or $\O(V)$, where $V = V_n(q)$. Write $G = \Cl_n(q)$ to cover both cases. As before, for a semisimple element $g \in G$, define $\nu(s) = \hbox{supp}(g)$, the codimension of the largest eigenspace of $g$ over $\bar \F_q$.

For $n<10$, Theorem \ref{main1} can be easily proved by exactly the same method of proof of \cite[Theorem 2]{LST} (improving the constant $D$ in Lemma 2.3 of \cite{LST} by using better bounds for $|G|$ and $|C_G(g)|_p$). 
So assume from now on that $n\ge 10$, so that the character ratio bounds in Propositions \ref{rat-so21} and \ref{rat-sp-so22} apply.

We begin with a lemma analogous to \cite[Lemma 3.2]{LST}.

\begin{lem}\label{sest} For $1\le s<n$, define 
$$N_s(G) := \{g\in \GSS : \nu(g)=s\}$$ 
and let $n_s(G):=|N_s(G)|$.
\begin{itemize}
\item[{\rm (i)}] If $g \in N_s(g)$ and $s<\frac{n}{2}$ then $|\CB_G(g)|_p < q^{\frac{1}{4}((n-s)^2+s^2) - v\frac{n-1}{2}}$, where $v=0$ or $1$ according as $G$ is symplectic or orthogonal.
\item[{\rm (ii)}] If $g \in N_s(g)$ and $s\ge \frac{n}{2}$ then $|\CB_G(g)|_p < q^{\frac{1}{4}(n^2-ns)}$.
\item[{\rm (iii)}] $\sum_{n-1 \geq s \geq n/2}n_s(G) < |G| < q^{\frac{1}{2}(n^2+n)-vn}$, where $v$ is as in $(ii)$.
\item[{\rm (iv)}] If $s < n/2$, then $n_s(G) < cq^{\frac{1}{2}s(2n-s+1)+\frac{n}{2}}$,
where $c$ is an absolute constant that can be taken to be $15.2$.
\end{itemize}
\end{lem}

\begin{proof}
(i) If $\nu(g)=s<\frac{n}{2}$, then the largest eigenspace of $g$ has dimension $n-s>\frac{n}{2}$, so has eigenvalue $\pm 1$, and so $\CB_G(g) \le \Cl_{n-s}(q) \times \Cl_s(q)$. Part (i) follows.

\vspace{2mm}
(ii) Now suppose $\nu(g) = s \ge \frac{n}{2}$, and let $E_\l$ ($\l \in \bar \F_q$) be an eigenspace of maximal dimension $n-s$. 

Assume first that $\l \ne \pm 1$. Then letting $a$ and $b$ denote the dimensions of the $+1$- and $-1$-eigenspaces, we have 
\begin{equation}\label{cent}
\CB_G(g) \le \prod_{i=1}^t \GL_{d_i}(q^{k_i}) \times \Cl_a(q) \times \Cl_b(q),
\end{equation}
where $n-s = d_1 \ge d_2\ge \cdots \ge d_t$ and also $d_1 \ge a\ge b$ and $2\sum_1^t k_id_i+a+b = n$. 
Hence $|\CB_G(g)|_p \le q^D$, where
\begin{equation}\label{expd}
D = \frac{1}{2}\sum_{i=1}^t k_id_i(d_i-1) + \frac{1}{4}(a^2+b^2).
\end{equation}

If $n\ge 4d_1$, this expression is maximised when $a=b=d_1$ and $(d_1,\ldots ,d_t) = (d_1,\ldots ,d_1,r)$ with $r\le d_1$ and $k_i=1$ for all $i$.. Hence in this case,
\[
D \le \frac{1}{2}(t-1)d_1(d_1-1) + \frac{1}{2}r(r-1) + \frac{1}{2}d_1^2 = \frac{1}{2}td_1^2-\frac{1}{2}(t-1)d_1+\frac{1}{2}r(r-1),
\]
and this is easily seen to be less than $\frac{1}{4}nd_1$, as required for part (ii).

Similarly, if $4d_1>n\ge 3d_1$, the expression (\ref{expd}) is maximised when $t=1$, $k_1=1$, $a=d_1$ and $b=r < d_1$; and when $3d_1>n \ge 2d_1$ (note that $n\ge 2d_1 = 2(n-s)$ by our assumption that $\nu(g) = s \ge \frac{n}{2}$), the expression (\ref{expd}) is maximised when $t=1$ and $a=r< d_1$. In each case, we see that $D< \frac{1}{4}nd_1$ as above.

Assume finally that the eigenvalue $\l = \pm 1$. In this case the centralizer $\CB(g)$ is as in (\ref{cent}), with $n-s=a \ge d_1\ge \cdots \ge d_t$ and also $a\ge b$ and $2\sum_1^t k_id_i+a+b = n$. Again we have $|\CB_G(g)|_p \le q^D$, with $D$ as in (\ref{expd}), and we argue as above that $D < \frac{1}{4}na = \frac{1}{4}n(n-s)$. This completes the proof of (ii).

\vspace{2mm}
(iii) This is clear.

\vspace{2mm}
(iv) If $\nu(g) = s < \frac{n}{2}$ then as in (i), the largest eigenspace of $g$ has eigenvalue $\pm 1$, so we have 
$\CB_G(g) \ge \Cl_{n-s}(q) \times T_s$, where $T_s$ is a maximal torus of $\Cl_s(q)$. Hence $|g^G| \le |G:\Cl_{n-s}(q)T_s| \le q^{\frac{1}{2}s(2n-s+1)}$. Also the number of conjugacy classes in $G$ is at most $15.2q^{n/2}$ by \cite{FG}, and (iv) follows. 
\end{proof}

\begin{lem}\label{stein} Let $\c \in \{\a,\b,\g_i\}$, where $\a,\b,\g_i$ are the irreducible characters of $G$ defined in Section \ref{red}. Then $\St \subseteq \c^{4n}$.
\end{lem}

\begin{proof} 
As in the proof of \cite[Lemma 2.3]{LST}, there are signs $\e_g=\pm 1$ such that 
\begin{equation}\label{useag}
\begin{array}{ll}
[\c^l,\St]_G & = \dfrac{1}{|G|}\sum_{g\in \GSS} \e_g \c^l(g)|\CB_G(g)|_p \\
               & = \dfrac{\chi^l(1)}{|G|}\left(|G|_p + \sum_{1 \neq g \in \GSS} \e_g \left(\frac{\chi(g)}{\chi(1)}\right)^l|\CB_G(g)|_p\right).
\end{array}
\end{equation}
 Hence $[\c^l,\St]_G \ne 0$ provided $\Sigma_l < |G|_p$, where
\[
\Sigma_l := \sum_{1 \neq g\in \GSS}  \left|\frac{\chi(g)}{\chi(1)}\right|^l|\CB_G(g)|_p.
\]
By Propositions  \ref{rat-so21} and \ref{rat-sp-so22}, if $s = \nu(g)$ we have 
\[
\frac{|\c(g)|}{\c(1)} \le \frac{1}{q^{s/3}}.
\]
Hence applying Lemma \ref{sest}, we have $\Sigma_l \le \Sigma_1+\Sigma_2$, where
\[
\begin{array}{l}
\Sigma_1 = \sum_{1\le s<\frac{n}{2}} cq^{\frac{1}{2}s(2n-s+1)+\frac{n}{2}}. \frac{1}{q^{ls/3}} .  q^{\frac{1}{4}((n-s)^2+s^2) - v\frac{n-1}{2}}, \\
\Sigma_2 =  \sum_{\frac{n}{2}\le s < n} q^{\frac{1}{2}(n^2+n)-vn}. \frac{1}{q^{ls/3}} . q^{\frac{1}{4}(n^2-ns)}.
\end{array}
\]
For a term in $\Sigma_1$, the exponent of $q$ is 
\[
\frac{1}{4}n^2-v\frac{n-1}{2} + \frac{1}{2}s(n+1)+\frac{1}{2}n-\frac{ls}{3}.
\]
As $|G|_p \le q^{\frac{1}{4}n^2-v\frac{n-1}{2}}$, taking $l=4n$ this gives
\[
\begin{array}{ll}
\frac{\Sigma_1}{|G|_p} & \le \sum_{1\le s<\frac{n}{2}} cq^{\frac{1}{2}s(n+1)+\frac{n}{2}-\frac{ls}{3}} \\
                                    & \le \sum_{1\le s<\frac{n}{2}} cq^{\frac{1}{2}n(1-\frac{5s}{3})+\frac{s}{2}}.
\end{array}
\]
Recalling that $c=15.2$, it follows that $\frac{\Sigma_1}{|G|_p} < \frac{1}{2}$ (except for $q=2, n\le 20$, in which case we obtain the same conclusion using slightly more refined estimates instead of Lemma \ref{sest}(iv)). 

For a term in $\Sigma_2$, the exponent of $q$ is 
\[
\frac{1}{2}(n^2+n)-vn +\frac{1}{4}n(n-s) - \frac{ls}{3},
\]
and leads similarly to the inequality $\frac{\Sigma_2}{|G|_p} < \frac{1}{2}$  when $l=4n$. 

We conclude that  $\Sigma_l < |G|_p$ for $l=4n$, proving the lemma.
\end{proof}

\begin{proof}[Proof of Theorem \ref{main1}] 
Let $1\ne \psi \in \Irr(G)$. By Lemma \ref{stein} together with Lemmas \ref{mc-r2} and \ref{mc-r3}, we have $\St \subseteq \psi^{8n}$ for $G = Sp_n(q)$, and 
$\St \subseteq \psi^{16n}$ for $G = \O^\e_n(q)$. Since $\St^2$ contains all irreducible characters by \cite{HSTZ}, the conclusion of Theorem \ref{main1} follows. 
\end{proof}

\section{Alternating groups: Proof of Theorem \ref{main2}}\label{pfth2}

In this section we prove Theorem \ref{main2}.

\begin{lem}\label{staircase}
Let $n := m(m+1)/2$ with $m \in \Z_{\geq 6}$, and let $\chi_m := \chi^{(m,m-1,\ldots,1)}$ be the staircase character of $\SSS_n$. Then
$$\chi_m(1) \geq |\SSS_n|^{5/11}.$$ 
\end{lem}

\begin{proof}
We will proceed by induction on $m \geq 6$. The induction base $m=6,7$ can be checked directly. For the induction step going from 
$m$ to $m+2$, note by
the hook length formula that $\chi_m(1)= n!/H_m$, where $H_m$ is the product of all the hook lengths in the Young diagram of 
the staircase partition $(m,m-1, \ldots,1)$. Hence it is equivalent to to prove that 
$$(m(m+1)/2)! > H_m^{11/6}.$$
Since the statement holds for $m$ and $H_{m+2}/H_m = (2m+3)!!(2m+1)!!$, it suffices to prove that 
\begin{equation}\label{eq:st1}
  \prod^{2m+3}_{i=1}(m(m+1)/2+i) > \bigl((2m+3)!!(2m+1)!!\bigr)^{11/6}
\end{equation}
for any $m \geq 6$. Direct computation shows that \eqref{eq:st1} holds when $3 \leq m \leq 40$. When $m \geq 40$, note that 
$$\begin{array}{ll}
   \prod^{2m+3}_{i=1}(m(m+1)/2+i) & > \bigl(m(m+1)/2+1\bigr)^{2m+3}\\ 
   & > \bigl((m+3)^{m+1}(m+2)^m\bigr)^{11/6}\\ 
   & > \bigl((2m+3)!!(2m+1)!!\bigr)^{11/6},\end{array}$$
proving \eqref{eq:st1} and completing the induction step.   
\end{proof}

\begin{proof}[Proof of Theorem \ref{main2}]
We will make use of \cite[Theorem 1.4]{S} which states that there exists an effective absolute constant $C_1 \geq 2$ such that 
\begin{equation}\label{eq:s}
  \chi^{t} \mbox{ contains }\Irr(\SSS_n) \mbox{ whenever }t \geq C_1n\log(n)/\log(\chi(1))
\end{equation}
for every non-linear $\chi \in \Irr(\SSS_n)$. With this, we will prove that when $n$ is sufficiently large we have 
\begin{equation}\label{eq:main2}
  \varphi^{k} \mbox{ contains }\Irr(\AAA_n) \mbox{ whenever }k \geq Cn\log(n)/\log(\varphi(1))
\end{equation}
for every nontrivial $\varphi \in \Irr(\AAA_n)$, with $C=5C_1^2$.

\smallskip
(i) Consider any $n \geq 5$ and any nontrivial $\varphi \in \Irr(\AAA_n)$. If $\varphi$ extends to $\SSS_n$, then we are done by 
\eqref{eq:s}. Hence we may assume that $\varphi$ lies under some $\chi^\l \in \Irr(\SSS_n)$, where $\l \vdash n$ is 
self-associate, and that $n$ is sufficiently large. By \cite[Proposition 4.3]{KST}, the latter implies that 
\begin{equation}\label{eq:a1}
  \varphi(1) \geq 2^{(n-5)/4}.
\end{equation}
Consider the Young diagram $Y(\l)$ of $\l$, and let $A$ denote the removable node in the last row of $Y(\l)$. 
Also let $\rho:=\chi^{\l \smallsetminus A} \in \Irr(\SSS_{n-1})$. Since $\l \smallsetminus A$ is not self-associate, 
$\rho$ is also irreducible over $\AAA_{n-1}$. Furthermore, by Frobenius' reciprocity,
$$1 \leq [(\chi^\l)|_{\SSS_{n-1}},\rho]_{\SSS_{n-1}} = [\chi^\l,\Ind^{\SSS_n}_{\SSS_{n-1}}(\rho)]_{\SSS_n},$$
whence $2\varphi(1) = \chi^\l(1) \leq \Ind^{\SSS_n}_{\SSS_{n-1}}(\rho)(1) = n\rho(1)$, and so
\begin{equation}\label{eq:a2}
  \rho(1) \geq (2/n)\varphi(1).
\end{equation}   
It follows from \eqref{eq:a1} and \eqref{eq:a2} that when $n$ is large enough, 
$$\log(\rho(1)) \geq \log(\varphi(1))-\log(n/2) \geq (9/10)\log(\varphi(1)).$$
Now we consider any integer 
\begin{equation}\label{eq:a3}
  s \geq \frac{10C_1}{9} \cdot \frac{n\log(n)}{\log(\varphi(1))}.
\end{equation}
This ensures that $s \geq C_1(n-1)\log(n-1)/\log(\rho(1))$, and so,
by \eqref{eq:s} applied to $\rho$, 
$\rho^s$ contains $\Irr(\SSS_{n-1})$.

\smallskip
(ii) Next, we can find a unique $m \in \Z_{\geq 3}$ such that 
\begin{equation}\label{eq:a4}
  n_0:=m(m+1)/2 \leq n-3 < (m+1)(m+2)/2,
\end{equation}
and consider the following partition
\begin{equation}\label{eq:a5}
  \mu:= (n-1-m(m-1)/2,m-1,m-2, \ldots,2,1)
\end{equation}  
of $n-1$. Note that $\mu$ has $m$ rows, with the first (longest) row 
$$\mu_1=n-1-m(m-1)/2 \geq m+2$$ 
by \eqref{eq:a4}. Hence, if $B$ is any addable 
node for the Young diagram $Y(\mu)$ of $\mu$, $Y(\mu \sqcup B)$ has at most $m+1$ rows and at least $m+2$ columns, and 
so is not self-associate. It follows that, for any such $B$, the character $\chi^{\mu \sqcup B}$ of $\SSS_n$ is irreducible over $\AAA_n$.

\smallskip
(iii) Recall that $\chi^\l|_{\AAA_n} = \varphi+\varphi^\star$ with $\varphi^\star$ being $\SSS_n$-conjugate to $\varphi$. 
It suffices to prove \eqref{eq:main2} for an $\SSS_n$-conjugate of $\varphi$. As $\chi^\l|_{\SSS_{n-1}}$ contains 
$\rho=\chi^{\l \smallsetminus A}$ which is irreducible over $\AAA_{n-1}$, without loss we may assume that $\varphi|_{\AAA_{n-1}}$ 
contains $\rho|_{\AAA_{n-1}}$. By the result of (i), $\rho^s$ contains $\chi^\mu$, with $\mu$ defined in \eqref{eq:a5} Thus 
\begin{equation}\label{eq:a6}
  1 \leq [\varphi^s|_{\AAA_{n-1}},(\chi^\mu)|_{\AAA_{n-1}}]_{\AAA_{n-1}}=\bigl[\varphi^s,
    \Ind^{\AAA_n}_{\AAA_{n-1}}\bigl((\chi^\mu)|_{\AAA_{n-1}}\bigr)\bigr]_{\AAA_n}.
\end{equation}    
Also recall that $\chi^\mu$ is an $\SSS_{n-1}$-character and $\SSS_n = \AAA_n\SSS_{n-1}$. Hence
$$\Ind^{\AAA_n}_{\AAA_{n-1}}\bigl((\chi^\mu)|_{\AAA_{n-1}}\bigr) = \bigl(\Ind^{\SSS_n}_{\SSS_{n-1}}(\chi^\mu)\bigr)|_{\AAA_{n}}.$$ 
Next,
$$\Ind^{\SSS_n}_{\SSS_{n-1}}(\chi^\mu) = \sum_{B~{\rm \tiny{addable}}}\chi^{\mu \sqcup B},$$
where, as shown in (ii), each such $\chi^{\mu \sqcup B}$ is irreducible over $\AAA_n$. Hence, it now follows from \eqref{eq:a6} that
there is an addable node $B_0$ for $Y(\mu)$ that $\varphi^s$ contains $\psi|_{\AAA_n}$, with $\psi:=\chi^{\mu \sqcup B_0}$.

\smallskip
(iv) By the choice of $B_0$, $\psi|_{\SSS_{n-1}}$ contains $\chi^\mu$, whence $\psi(1) \geq \chi^\mu(1)$. Next, by \eqref{eq:a4}, we 
can remove $n-1-n_0 \geq 2$ nodes from the first row to arrive at the staircase partition $(m,m-1, \ldots,1) \vdash n_0$. In particular,
$\psi|_{\SSS_{n_0}}$ contains the character $\chi_m$ of $\SSS_{n_0}$. By Lemma \ref{staircase}, for $n$ sufficiently large we have
\begin{equation}\label{eq:a7}
  \log(\psi(1)) \geq \log(\chi_m(1)) \geq (5/11)\log(n_0!) \geq (2/5)n\log(n),
\end{equation}  
since 
$$n_0 = m(m+1)/2 \geq n-(m+2) \geq n-(3/2+\sqrt{2n-4})$$
by the choice \eqref{eq:a4} of $m$. Now we consider the integer $t := \lceil (5/2)C_1 \rceil \leq 3C_1$ (since $C_1 \geq 2$). Then 
$$C_1n\log(n)/\log(\psi(1)) \leq (5/2)C_1 \leq t$$
by \eqref{eq:a7}, and so $\psi^t$ contains $\Irr(\SSS_n)$ by \eqref{eq:s} applied to $\psi$. In particular,
$(\psi^t)|_{\AAA_n}$ contains $\Irr(\AAA_n)$.

Recall from (iii) that $\varphi^s$ contains the irreducible character $\psi|_{\AAA_n}$. It follows that $\varphi^{st}$ 
contains $(\psi^t)|_{\AAA_n}$, and so $\varphi^{st}$ contains $\Irr(\AAA_n)$.

\smallskip
(v) Finally, consider any integer $k \geq Cn\log(n)/\varphi(1)$ with $C=5C_1^2$. Then 
$$k/t \geq k/3C_1 \geq (5/3)C_1n\log(n)/\log(\varphi(1)).$$
As $C_1\geq 1$ and $n\log(n)/\log(\varphi(1) \geq 2$, we have that 
$$(5/3-10/9)C_1n\log(n)/\log(\varphi(1)) \geq 10/9.$$
In particular, we can find an integer $s_0$ such that 
$$k/t \geq s_0 \geq (10/9)C_1n\log(n)/\log(\varphi(1)).$$
As $s$ satisfies \eqref{eq:a3}, the result of (iv) shows that $\varphi^{s_0t}$ contains $\Irr(\AAA_n)$.

Now, given any $\gamma \in \Irr(\AAA_n)$, we can find an irreducible constituent $\delta$ of $\varphi^{k-s_0t}\overline\gamma$. 
By the previous result, $\varphi^{s_0t}$ contains $\overline\delta$. It follows that $\varphi^k$ contains 
$\varphi^{k-s_0t}\overline\delta$, and 
$$[\varphi^{k-s_0t}\overline\delta,\gamma]_{\AAA_n}= [\varphi^{k-s_0t}\overline\gamma,\delta]_{\AAA_n} \geq 1,$$
i.e. $\varphi^k$ contains $\gamma$, and the proof of \eqref{eq:main2} is completed.
\end{proof}

\section{Products of characters}\label{pfth3}

\subsection{Products of characters in classical groups}

This is very similar to the proof of Theorem 2 of \cite{LST}. Let $G = G_r(q)$ be a simple group of Lie type of rank $r$ over 
$\F_q$.

\begin{lem}\label{stdiam}
There is an absolute constant $D$ such that for any $m\ge Dr^2$ and any $\c_1,\ldots,\c_m \in {\rm Irr}(G)$, we have
$[\prod_1^m\c_i,\St]_G\ne 0$. Indeed, $D=163$ suffices.
\end{lem}

\begin{proof}
This is proved exactly as for \cite[Lemma 2.3]{LST}, replacing the power $\c^m$ by the product $\prod_1^m\c_i$.
\end{proof}

\begin{proof}[Proof of Theorem \ref{rodsax}(i)]
Take $c_1=3D$ with $D$ as in the lemma, and let $\c_1,\ldots ,\c_l \in {\rm Irr}(G)$ with $l=c_1r^2$. Writing $m=l/3 = Dr^2$, Lemma \ref{stdiam} shows that each of the products $\prod_1^m\c_i$, $\prod_{m+1}^{2m}\c_i$ and $\prod_{2m+1}^{3m}\c_i$ contains $\St$. Hence $\prod_1^l\c_i$ contains $\St^3$, and this contains ${\rm Irr}(G)$ by \cite[Prop. 2.1]{LST}. This completes the proof.
\end{proof}

\subsection{Products of characters in linear and unitary groups}

This is similar to the proof of Theorem 3 of \cite{LST}. Let $G = \PSL_n^\e(q)$.

We shall need \cite[Theorem 3.1]{LST}, which states that there is a function $f:\N\to \N$ such that for any $g \in \GSS$ with $s = \nu(g)$, and any $\c \in {\rm Irr}(G)$, we have
\begin{equation}\label{31lst}
|\c(g)| < f(n)\c(1)^{1-\frac{s}{n}}.
\end{equation}

Again we begin with a lemma involving the Steinberg character.

\begin{lem}\label{ste} Let $m\in \N$ and let $\c_1,\ldots,\c_m \in {\rm Irr}(G)$. Set $c=44.1$, and define
\[
\begin{array}{l}
\D_{1m} = cf(n)^m \sum_{1\le s <n/2} q^{ns+\frac{3n}{2}-1}\left(\prod_1^m\c_i(1)\right)^{-s/n},\\
\D_{2m} = f(n)^m \sum_{n/2\le s<n}q^{n^2-\frac{1}{2}n(s-1)-1}\left(\prod_1^m\c_i(1)\right)^{-s/n}.
\end{array}
\]
If $\D_{1m}+\D_{2m}<1$, then $[\prod_1^m\c_i,\,\St]_G \ne 0$.
\end{lem}

\begin{proof}
Arguing as in the proof of \cite[Lemma 3.3]{LST}, we see that $[\prod_1^m\c_i,\,\St]_G \ne 0$ provided $\D_m <1$, where
\[
\D_m := \sum_{1 \leq s < n/2} cq^{ns+\frac{3n}{2}-1}\left|\prod_1^m\frac{\chi_i(g_{i,s})}{\chi(1)}\right| +
\sum_{n/2 \leq s < n} q^{n^2-\frac{1}{2}n(s-1)-1}\left|\prod_1^m\frac{\chi(g_{i,s})}{\chi(1)}\right|,
\]
where $g_{i,s} \in \GSS$ is chosen such that $\nu(g_{i,s})=s$ and $|\c_i(g_{i,s})|$ is maximal. Now application of (\ref{31lst}) gives the conclusion.
\end{proof}

\begin{lem}\label{better} There is a function $g:\N\to \N$ such that the following holds. Suppose that $\c_1,\ldots,\c_m \in {\rm Irr}(G)$ satisfy $\prod_1^m \c_i(1) > |G|^3$. Then provided $q>g(n)$, we have $[\prod_1^m\c_i,\,\St]_G \ne 0$.
\end{lem}

\begin{proof}
We have $|G|>\frac{1}{2}q^{n^2-2}$, so for $s<n$,
\[
\left(\prod_1^m\c_i(1)\right)^{-s/n} < 8q^{-3ns+\frac{6s}{n}}.
\]
Hence 
\[
\D_{1m} \le 8cf(n)^m \sum_{1\le s <n/2} q^{-2ns+\frac{3n}{2}+2},
\]
and
\[
\begin{array}{ll}
\D_{2m} & \le 8f(n)^m \sum_{n/2\le s<n}q^{n^2-\frac{1}{2}n(s-1)-1} q^{-3ns+6} \\
              & \le 8f(n)^m \sum_{n/2\le s<n}q^{-\frac{3n^2}{4}+\frac{1}{2}n+5}.
\end{array}
\]
Now the conclusion follows from Lemma \ref{ste} (using some slight refinements of the above inequalities for $n\le 4$). 
\end{proof} 

\begin{proof}[Proof of Theorem \ref{rodsax}(ii)] 
Assume  $\c_1,\ldots,\c_l \in {\rm Irr}(G)$ satisfy $\prod_1^l \c_i(1) > |G|^{10}$. Since $\c_i(1) < |G|^{1/2}$ for all $i$, there are disjoint subsets $I_1,I_2,I_3$ of $\{1,\ldots ,m\}$ such that $\prod_{i\in I_k} \c_i(1) > |G|^3$ for $k=1,2,3$. Then  $\prod_{i\in I_k} \c_i$ contains $\St$ for each $k$, by Lemma \ref{better}, and so $\prod_1^l\c_i$ contains $\St^3$, hence contains ${\rm Irr}(G)$, completing the proof.
\end{proof}

\subsection{Products of characters in symmetric and alternating groups}
\begin{prop}\label{rs2-an}
Let $G \in \{\SSS_n,\AAA_n\}$, $l \in \Z_{\geq 1}$, and let $\chi_1,\chi_2, \ldots,\chi_l \in \Irr(G)$ with $\chi_i(1) > 1$ for all $i$.
\begin{enumerate}[\rm(i)]
\item If $l \geq 8n-11$, then $\bigl(\prod^l_{i=1}\chi_i\bigr)^{2}$ contains $\Irr(G)$.
\item Suppose that, for each $1 \leq i \leq l$, there exists some $j \neq i$ such that $\chi_j = \chi_i$. If $l \geq 24n-33$ then
$\prod^l_{i=1}\chi_i$ contains $\Irr(G)$.
\end{enumerate}
\end{prop}

\begin{proof}
(i) Let $\chi^\l$ denote the irreducible character of $\SSS_n$ labeled by the partition $\l \vdash n$. A key result established in the proof of 
\cite[Theorem 5]{LST} is that, for any $i$ there exists 
$$\a_i \in \left\{\chi^{(n-1,1)},\chi^{(n-2,2)},\chi^{(n-2,1^2)},\chi^{(n-3,3)}\right\}$$
such that $\chi_i^2$ contains $(\a_i)|_G$. Since $l \geq 8n-11$, there must be some 
$$\b \in \left\{\chi^{(n-1,1)},\chi^{(n-2,2)},\chi^{(n-2,1^2)},\chi^{(n-3,3)}\right\}$$
such that $\b=\a_i$ for at least $2n-2$ distinct values of $i$. It follows that $\bigl(\prod^l_{i=1}\chi_i\bigr)^{2}=\g\d$, 
where $\g := \b^{2n-2}|_G$, and $\d$ is a character of $G$. By \cite[Theorem 5]{LST}, $\b^{2n-2}$ contains 
$\Irr(\SSS_n)$, whence $\g$ contains $\Irr(G)$. Now the arguments in the last paragraph of the proof of 
Theorem \ref{main2} show that $\g\d$ contains $\Irr(G)$ as well.

\smallskip
(ii) Note that the assumptions imply, after a suitable relabeling, that $\prod^l_{i=1}\chi_i$ contains $\sigma\l$, 
where $\l$ is a character of $G$ and 
$$\s= \prod^{8n-11}_{i=1}\chi_i^2.$$
(Indeed, any subproduct $\chi_{i_1}\ldots\chi_{i_t}$ with $t> 1$ and $\chi_{i_1}=\ldots =\chi_{i_t}$ yields a term 
$(\chi_{i_1}^2)^{\lfloor t/2 \rfloor}$ in $\s$.) By (i), $\s$ contains $\Irr(G)$, and so we are done as above.
\end{proof}

\end{document}